\documentclass{amsart}
\usepackage{amsmath, amssymb, amsthm}

\usepackage{epsfig}
\usepackage{graphicx}
\usepackage{color}

\usepackage{pictexwd, dcpic}
 
\theoremstyle{theorem}
\newtheorem{theorem}{Theorem}

\theoremstyle{lemma}
\newtheorem{lemma}{Lemma}[section]

\theoremstyle{prop}
\newtheorem{prop}[lemma]{Proposition}

\theoremstyle{definition}
\newtheorem{definition}[lemma]{Definition}

\theoremstyle{corollary}
\newtheorem{corollary}[lemma]{Corollary}

\theoremstyle{remark}
\newtheorem{remark}[lemma]{Remark}
\newcommand{\A}{\mathbb{A}}
\newcommand{\AF}{\mathcal{A}}
\newcommand{\Z}{\mathbb{Z}}
\newcommand{\B}{\mathcal{B}}
\newcommand{\X}{\mathcal{X}}
\newcommand{\Y}{\mathcal{Y}}
\newcommand{\R}{\mathbb{R}}

\newcommand{\F}{\mathcal{F}}
\newcommand{\PR}{\mathbb{P}}
\newcommand{\h}{h}
\newcommand{\Aut}{\text{Aut}}

\begin{document}

\title{Automorphisms of Products of Drinfeld Half Planes}
\author{Gil Alon}
%\email{gil.alon@gmail.com}
\address{Department of Mathematics and Computer Science, The Open University of Israel, P.O. Box 808, Raanana, Israel}
%\classification{14G22}

\begin{abstract}
The Drinfeld upper half-planes play the role of symmetric spaces in the $p$-adic analytic world. We find the automorphism group of a product of such spaces, where each may be defined over a different field. We deduce a rigidity theorem for quotients of such products by discrete and torsion free groups. 
\end{abstract}
\maketitle
% XXXXXXXXXXXXXXXXXXXXXXXXXXXXXXXXXXXXXXXXXXXXXXXXXXXXXXXXXXXXXXXXXXXXXXXXXXXXXXXXXXXXXXX
\sloppy

\section{Introduction}
A theorem of E. Cartan (\cite{helg}, IV 3.2 and II 2.6) states that the isometry group of a real symmetric space has a canonical structure of a Lie group. Over a non-Archimedean field, the only known analogs of symmetric spaces are the Drinfeld symmetric spaces and their products. In analogy to bounded symmetric domains, products of Drinfeld spaces provide non-Archimedean uniformization for Shimura varieties, as shown by Varshavsky and Rapoport-Zink (see \cite{rap}, \cite{var}). The purpose of this paper is to determine the group of analytic automorphisms of such a product. 

We let $l$ be a non-Archimedean local field, $K$ a non-Archimedean field containing $l$, and $k_1,..,k_r$ finite extensions of $l$, which are contained in $K$. We will consider the $K$-analytic space (in the sense of Berkovich, see \cite{ber2}, \cite{ber4}) $\prod_{i=1}^r\Omega_{k_i,K}^{d_i}$, where $\Omega_{k,K}^d=\Omega_k^d\hat{\otimes}K$, and $\Omega_k^d$ is the Drinfeld space of dimension $d$ over $k$, defined as the open subset of the projective space $\mathbb{P}^{d}_k$ obtained by removing all the $k$-rational hyperplanes. The main result of the paper is the following theorem:

\begin{theorem} \label{T}
Let $M$ be the permutation group $\{\sigma \in S_r: k_{\sigma(i)}=k_i \text{ and } d_{\sigma(i)}=d_i \text{ for } i=1,..,r\}$. Then we have
$$ \Aut_K(\prod_{i=1}^r\Omega_{k_i,K}^{d_i})=M \ltimes \prod_{i=1}^rPGL_{d_i+1}(k_i). $$
\end{theorem}

The case $r=1$ of theorem \ref{T} was proved in \cite{ber} and also in \cite{ikato} in the zero characteristic case. 

The proof in  \cite{ber} and its extension here to the case of a product demonstrate the power and usefulness of Berkovich spaces. For example, there is a natural,   $PGL_{d+1}(k)$-equivariant surjective map from the Drinfeld symmetric space $\Omega_{k,K}^{d}$ to the Bruhat-Tits building of $SL_{d+1}/k$. This well known map exists in the rigid analytic sense as well. But in the framework of Berkovich spaces, this map has a right inverse: The building is naturally embedded in the symmetric space. The points of the building (which is a Euclidean object) inside the $p$-adic symmetric space do not exist in the rigid analytic sense, but exist in the Berkovich sense. This property is extremely helpful in finding the automorphism group of the symmetric space.

Let us overview the proof of theorem \ref{T}, and start by recalling the strategy of the proof for the case $r=1$, in \cite{ber}: Given an automorphism $\phi$ of $\X:= \Omega_{k,K}^d$, it is shown that the Bruhat-Tits building $\B$ of $SL_{d+1}/k$ (viewed as a subset of $\X$) is preserved by $\phi$, and moreover, that $\phi$ induces a simplicial autopmorphism of $\B$. Then, $\phi$ is composed with an element $g$ of $G:=PGL_{d+1}(k)$, such that $\phi':=g\phi$ fixes an apartment $\Lambda$ of $\B$ pointwise. It is then shown that $\X$ can be projected onto $\Lambda$, and that $\phi'$ commutes with this projection, hence $\phi'$ preserves the fibers of the projection.  This, together with a lemma on bounded holomorphic functions on $\X$, enables to determine the automorphism $\phi$. 

Let us remark that there is a gap in the above proof. There are simplicial automorphisms of $\B$ whose composition with any element $G$ does not fix any chamber $\B$ pointwise (see section \ref{labelling}). Hence, we can not assume that $\phi$ fixes an apartment pointwise. In the proof presented here, we will overcome this problem by indroducing a labelling on the vertices of $\B$ and showing that an automorphism of $\B$ is either \emph{label preserving}, in which case its restriction to an apartment is induced by an element of $G$, or \emph{label reversing}, in which case it can be composed with a certain involution of $\B$ which we construct, and then it is made label preserving. Then, at the last stage of the proof we will show that in fact the second possibility can not occur. 

Let us now describe our proof strategy in the case where $r>1$ and all the fields $k_i$ are equal. We embed the product of the corresponding Bruhat-Tits buildings (which we still call $\B$) in our space. We show, as before, that $\phi$ preserves $\B$ and analyze the action of $\phi$ on a chamber of $\B$. A key ingredient in this analysis is the labelling of the building, and a counting argument which determines its interaction with $\phi$. From this analysis it follows that after a permutation on the coordinates, and composing with the above involutions in some of the coordinates, an apartment is fixed pointwise. Then, as before, we prove that $\phi$ acts on the fibers of the projection to the apartment. 

In the general case, the fields $k_i$ may not all be equal, and therefore we are not allowed to permute coodinates with different fields. We remedy this by embedding our building in a larger building, in which these permutations are allowed.

As a result of our main theorem, we prove the following rigidity theorem for quotients of such products by groups of automorphisms:
\begin{theorem} \label{R} If $k'_1,..,k'_s$ are finite extensions of $l$ which are contained in $K$, and $\Gamma_1$ and $\Gamma_2$ are discrete and torsion free subgroups of $\prod_i PGL_{d_i+1}(k_i)$ and $\prod_j PGL_{d'_j+1}(k'_j)$ respectively, such that the quotients $\Gamma_1 \setminus (\prod_{i=1}^r\Omega_{k_i,K}^{d_i})$ and  $\Gamma_2 \setminus (\prod_{i=1}^s\Omega_{k'_i,K}^{d'_i})$ are isomorphic as $K$-analytic spaces, then $r=s$, and there exists a permutation $\sigma \in S_r$ such that $k_i=k'_{\sigma(i)}$ and $d_i=d'_{\sigma(i)}$ for all $i$, and $\Gamma_1$ and $\{(g_{\sigma(1)},..,g_{\sigma(r)}):(g_1,..,g_r)\in \Gamma_2\}$ are conjugate in $\prod_{i=1}^rPGL_{d_i+1}(k_i)$. \end{theorem}

% XXXXXXXXXXXXXXXXXXXXXXXXXX  The Drinfeld space as a Berkovich space  XXXXXXXXXXXXXXXXXXXXXXX

\section{The Drinfeld space as a Berkovich space} \label{sec:DrSp}

Let $k$ be a non-Archimedean local field, $K$ a non-Archimedean field containing $k$, and $d$ be a natural number.

The affine space $\A^{d}_{K}$ is the Berkovich space whose points are the multiplicative seminorms on the polynomial ring $K[t_1,..,t_{d}]$ which extend the given norm on $K$. It is covered by affinoids of the form 

$$B(d,r)=M(K\{r^{-1}t_1,..,r^{-1}t_{d}\})$$ 

(See \cite{ber2} for the definitions of $M(\mathcal A)$, the spectrum of an affinoid algebra $\mathcal A$, and of the affinoid algebras of the form $K\{r_1^{-1}t_1,..,r_d^{-1}t_d\}$). The affinoids $B(d,r)$ form a net on $\A^{d}_{K}$ (as defined in \cite{ber4}), and thereby determine its analytic structure. 

The projective space $\PR^{d}_K$ is the Berkovich space whose points are multiplicative seminorms $\rho$ on the polynomial ring $K[T_0,..,T_{d}]$, which extend the given norm on $K$, such that $\rho(T_i)\neq 0$ for at least one $i$, modulo the following equivalence relation: two such multiplicative seminorms $\rho_1, \rho_2$ are equivalent if and only if there exists a positive real number $c$ such that for any homogenous polynomial $P\in K[T_0,..,T_{d}]$ of degree $e$, we have $\rho_1(P)=c^e \rho_2(P)$. Indeed, it is an easy exercise that this space is covered by the affine pieces defined by the conditions $T_i \neq 0$ for $i=0,..,d$.

For any $a=(\alpha_0,..,\alpha_{d})\in k^{d+1} \setminus \{0 \}$, let $H_a$ be the subset of $\mathbb P^{d}_K$ defined by the condition $\sum_{i=0}^{d}{\alpha_i T_i}=0$. In terms of the multiplicative seminorms involved, it is the set $$H_a= \big \{[\rho] \in \PR^{d}_K \,| \,   \rho(\sum_{i=0}^{d}{\alpha_i T_i})=0 \big \}$$

Let us now define the Drinfeld space $\Omega^{d}_{k,K}$. As a set, it is
$$  \Omega^{d}_{k,K} = \PR^{d}_K \setminus \bigcup_{a\in k^{d+1} \setminus \{0 \}}H_a $$ 

We will prove that $\Omega^{d}_{k,K}$ is an open subset of  $\PR^{d}_K$. This will automatically endow  $\Omega^{d}_{k,K}$ with the structure of a Berkovich space.

Note that $$  \Omega^{d}_{k,K} \subseteq \PR^{d}_K \setminus H_{(1,0,..,0)}\cong \A^{d}_K$$ Therefore, we can view $  \Omega^{d}_{k,K}$ as a subspace of $\A^{d}_K$. We will use this point of view whenever it is convenient, via the affine coordinates $t_i=\frac{T_i}{T_0}$ for $i=1,..,d$. It is also convenient to set $t_0=1$.

Let us define two coverings of $\Omega^{d}_{k,K}$ by increasing sequences of sets, the first by open sets and the second by affinoids. These coverings were introduced by Schneider and Stuhler in \cite{ss} in the rigid analytic setting. We let $\pi$ be a uniformizer of $k$, and define for any $n\geq 1$,

\begin{equation*}
\begin{split}
\Omega^{d}_{k,K}(n)= \Big \{   [\rho]\in \PR^d_K \, | \, & \rho (\sum_{i=0}^{d}\alpha_i T_i  ) > |\pi|^n   \max_i {|\alpha_i|}
\max_i {\rho(T_i)}  \\
 & \text{ for all } (\alpha_0,..,\alpha_{d})\in k^{d+1} \setminus \{0\} \Big \} 
\end{split}
\end{equation*}

and

\begin{equation*}
\begin{split}
\Omega^{d}_{k,K}[n]=\Big \{   [\rho]\in \PR^d_K \, | \, & \rho(\sum_{i=0}^{d}\alpha_i T_i) \geq |\pi|^n   \max_i {|\alpha_i|}
\max_i {\rho(T_i)}  \\
 & \text{ for all } (\alpha_0,..,\alpha_{d})\in k^{d+1} \setminus \{0\}\Big \} 
\end{split}
\end{equation*}

\begin{prop} \label{open} 
\begin{enumerate}
\item For all $n$, $\Omega^{d}_{k,K}[n]$ is a strictly affinoid domain in $\PR^{d}_K$.
\item For all $n$, $\Omega^{d}_{k,K}(n)$ is an open subset of $\PR^{d}_K$.
\item $\Omega^{d}_{k,K} = \bigcup_{n \geq 1} \Omega^{d}_{k,K}(n) =  \bigcup_{n \geq 1} \Omega^{d}_{k,K}[n]$.
\item $\Omega^{d}_{k,K}$ is an open subset of $\PR^{d}_K$.
\end{enumerate}
\end{prop}
\begin{proof}
\begin{enumerate}
\item
For a vector $a=(\alpha_0, .., \alpha_d) \in k^{d+1}$, let
$$H_n(a)=\Big \{ [\rho] \in \PR^d_K \, | \, \rho(\sum_{i=0}^{d}\alpha_i T_i) \geq |\pi|^n   \max_i {|\alpha_i|}\max_i {\rho(T_i)} \Big \} $$ 
Then $\Omega^d_{k,K}[n]=\bigcap_{a\in k^{d+1}\setminus \{0 \}} H_n(a)$.
Setting $a$ to be the unit vector $(1,0,..,0)$ we see that for any $x=[\rho]\in \Omega^d_{k,K}[n]$, $\rho(T_0) \geq |\pi^n|\rho(T_i)$ for all $1 \leq i \leq d$. Setting $a$ to be the $i^\text{th}$ unit vector, we see that $\rho(T_i) \geq |\pi^n| \rho(T_0)$. We conclude that $\Omega^d_{k,K}[n]$ is contained in the multi-annulus
$$A(n):=\Big \{ x \in \A^d_K \, | \,  |\pi|^n \leq |t_i(x)| \leq |\pi|^{-n} \text{ for all } i\Big \} $$

Let us call a vector $a=(\alpha_0,..,\alpha_d)$ \emph{unimodular} if $\max_i |\alpha_i|=1$. 
Since $H_n(a)=H_n(ca)$ for all scalars $c\in k \setminus \{0 \}$, it is enough to take unimodular vectors $a$ in the above intersection. Moreover, by lemma 1.2 in \cite{ss}, if two unimodular vectors $a$ and $b$ are equal modulo $\pi^n$, then $H_n(a)=H_n(b)$. Hence, $\Omega^d_{k,K}[n]$ is a \emph{finite} intersection $H(a_1) \cap .. \cap H(a_m)$. 
We can replace the condition in the definition of $H(a)$ by the conditions $\rho(\sum_{i=0}^{d}\alpha_i T_i) \geq |\pi|^n   (\max_i {|\alpha_i|} ){\rho(T_j)}$ for $j=0,..,d$. Hence, $\Omega^d_{k,K}[n]$ is a Laurent domain (see \cite{ber2}, remarks 2.2.2) in $A(n)$, defined by a finite set of conditions of the form $|(\sum_{i=0}^d \alpha_i t_i)/t_j| \geq |\pi^n|$. In particular, it is a strictly affinoid domain.
\item
Similarly to the previous item, we can replace the infinite set of conditions in the definition of $\Omega^{d}_{k,K}(n)$ by a finite one. Hence, $\Omega^{d}_{k,K}(n)$ is defined in  $\PR^d_K$  by a finite set of conditions of the form $\rho(\sum_i \alpha_i T_i) > c \rho(T_j)$. Hence, it is open.
\item
Let $[\rho] \in \Omega^d_{k,K}$, and let $U$ be the set of unimodular vectors in $k^{d+1}$. Since $U$ is compact and the function $$f((\alpha_0,..,\alpha_d))=\frac{\rho(\sum_{i=0}^{d}\alpha_i T_i)}{\max_i {|\alpha_i|}\max_i {\rho(T_i)}}$$ is continuous and positive on $U$, there is a positive lower bound on the values of $f$. Choosing $n$ such that $|\pi^n|$ is less than the lower bound, we have $[\rho] \in \Omega^{d}_{k,K}(n) \subseteq \Omega^{d}_{k,K}[n]$.
\item
This follows immediately from the previous two claims. \qedhere
\end{enumerate}\end{proof}

\begin{remark}
It follows from proposition \ref{open} that the above description of the Drinfeld space is compatible with its description elsewhere in the literature. In fact, under the fully faithful functor from Hausdorff strictly $K$-analytic Berkovich spaces to quasiseparated rigid analytic spaces (\cite[Theorem 1.6.1]{ber4}), the space $\Omega^{d}_{k,K}$ corresponds to the rigid analytic Drinfeld space as defined in \cite{ss}, since the covering by $\Omega^{d}_{k,K}(n)$ maps to the admissible affinoid covering defined in \cite{ss}. Hence, the calculation here of the automorphism group is valid in the rigid analytic sense, and the same is true for products of such spaces.
\end{remark}

Finally, let us note that the group $G=PGL_{d+1}(k)$ acts on $\PR^d_K$ by analytic automorphisms, and since we removed all the $k$-rational hyperplanes, this induces such an action of $G$ on  $\Omega^{d}_{k,K}$.

% XXXXXXXXXXXXXXXXXXXXXXXXXXXXXXXXXXX  The Bruhat-Tits building  XXXXXXXXXXXXXXXXXXXXXXXXXXXXX

\section{The Bruhat-Tits building}\label{sec:btb}

In this section we recall some definitions and facts on the Bruhat-Tits building of $SL_{d+1}/k$ and its relation to the Drinfeld space. More details and proofs can be found in \cite {ber} and \cite{brt}.

We will use the terminology introduced in \cite{brt}, I.1, regarding simplicial and polysimplicial complexes. Let us recall it briefly. A polysimplicial complex is a set $A$, together with a family $\mathcal{F}$ of subsets of $A$ (called faces), a partial ordering on $\mathcal{F}$ and an affine structure on $\bar{F}:=\bigcup\limits_{F' \leq F} F'$ for any $F\in \mathcal{F}$, such that:
\begin{enumerate}
\item
The sets of $\mathcal{F}$ form a partition of $A$.
\item
For any $F \in \mathcal{F}$, $\bar{F}$ is a closed polysimplex (i.e a finite product of simplices), in a compatible way with the affine structure and order relation.  We call each $\bar{F}$ a \emph{closed face}.
\item
The dimensions of the faces are bounded. We call a closed face of maximum dimension a \emph{chamber} and a closed face of codimension $1$ a \emph{facet}.
\item
Any two faces are connected by a sequence of chambers, such that consecutive terms intersect in a facet. Such a sequence is called a \emph{gallery}.
\end{enumerate}
We order the closed faces by inclusion. The closed faces are in order preserving bijection with the faces, since $F=\bar{F}^{\circ}$.
A polysimplicial complex has a topology, which is the gluing of the Euclidean topologies on the closed faces along the inclusion maps.
A morphism of  polysimplicial complexes is a map of the underlying sets, carrying faces to faces, and preserving the affine structure. We call such a morphism \emph{chambered} if, in addition, it carries chambers to chambers isomorphically. We call a $0$-dimensional closed face a \emph{vertex}, and a $1$-dimensional closed face an \emph{edge}.
A polysimplicial complex is called a simplicial complex if all its closed faces are simplices.

We keep the notation of the previous section, and let $G=PGL_{d+1}(k)$ and $V$ be the $d+1$-dimensional $k$-vector space,
$$ V=\bigoplus_{i=0}^d k T_i \subseteq K[T_0,..,T_d] $$ 

The building of $SL_{d+1}/k$ is a simplicial complex, which we denote by $B^d_k$. It has two well known descriptions: a combinatorial one, in which the vertices are given in terms of lattices, and a Euclidean one, in which the points are given in terms of norms. Let us start with the combinatorial one.

A lattice in $V$ is a finitely generated $O_k$-module which spans $V$ over $k$. Two lattices $L_1,L_2$ in $V$ are called \emph{homothetic} if $L_1 =\alpha  L_2$ for some $\alpha \in k \setminus \{0\}$. The vertices of $B^d_k$ are the homothety classes of lattices in $V$. A face of $B^d_k$ (given by the set of its vertices) is a set of the form $\{[L_0],..,[L_k]\}$ for distinct lattices $L_0,..,L_k$ satisfying $L_0 \supseteq L_2 \supseteq .. \supseteq L_k \supseteq \pi L_0$.
The chambers are of dimension $d+1$. A standard example of such a chamber is the chamber $\{[L_0],...,[L_d]\}$ given by $$L_i  = \langle T_0,..T_{d-i},\pi T_{d+1-i} ,.., \pi T_d \rangle $$ We shall call this chamber \emph{the basic chamber} of $B^d_k$.

The second description of $B^d_k$ is via \emph{norms} on $V$. A function $\rho:V\rightarrow \R$ is called a norm if it satisfies: 
\begin{enumerate}
\item $\rho(u+v)\leq \rho(u)+\rho(v)$ for all $u,v\in V$.
\item $\rho(\alpha v)=|\alpha|\rho(v)$ for all $v\in V, \alpha \in k$.
\item $\rho(v)>0$ for all $v \in V \setminus \{0\}$.
\end{enumerate}

Two norms $\rho_1,\rho_2$ on $V$ are called homothetic if $\rho_1=c\cdot \rho_2$ for some constant $c>0$. Then the underlying set of $B^d_k$ (which we denote by the same notation) is the set of norms on $V$ modulo homothety. 

For any basis $(v_0,..,v_d)$ of $V$, let $A(v_0,..,v_d)$ be the set of homothety classes of norms of the form 

$$ \rho(\sum \alpha_i v_i) = \max_i {c_i |\alpha_i|} $$

for some positive reals $c_0,..,c_d$. Such sets are called \emph{apartments} of $B^d_k$, and each is isomorphic to $\R^{d+1}/\R(1,1,..,1)$ via the map from $A(v_0,..,v_d)$ to $\R^{d+1}$, 

$$[\rho]\mapsto [(-\log (\rho(v_0)), .., -\log(\rho(v_d)))]$$

Therefore, the apartments are Euclidean spaces. The apartments cover the building $B^d_k$, and moreover, for any two points in the building there is an apartment containing 
both. 

We will denote the apartment of the standard basis, $A(T_0,..,T_d)$ by $\Lambda$. We call this apartment \emph {the basic apartment} of $B^d_k$.

The combinatorial description of $B^d_k$ and the Euclidean one are related by the following correspondence:  The class of a lattice $L$ corresponds to the homothety class of the norm $\rho_L$, defined by

\begin{equation}
\label{norm_formula}
\rho_L(v)=\min \{ |c| \, | \, c\in k \setminus \{0\}\text{ and } c^{-1}v \in L \}
\end{equation}
for any $v \in V \setminus\{0\}$. In particular, the basic chamber described above is given in terms of norms by $$ \Delta = \{\rho \in \Lambda \, | \, \rho(T_0) \geq \rho(T_1) \geq .. \geq \rho(T_d) \geq |\pi|\rho(T_0) \} $$

Let us define the maps discussed in the introduction,
\[ \begin{array}{c}
\Omega^d_{k,K} \\
j\uparrow \downarrow \tau \\
B^d_k
\end{array}\]
Tha map $\tau$ is simply defined by restriction: Given a point $[\rho]\in \Omega^d_{k,K}$, the multiplicative seminorm $\rho$ on $K[T_0,..,T_d]$ is nonzero on each nonzero element of $V$, and therefore the restriction $\rho |_V$ is a norm on $V$. We define $\tau([\rho])=[\rho|_V]$.

The map $j$ is defined as follows: Given a norm $\rho$ on $V$, let $(v_0,..,v_d)$ be any basis of $V$ such that $[\rho] \in A(v_0,..,v_d)$. We extend $\rho$ to a seminorm $\tilde{\rho}$ on $K[T_0,..,T_d]$ by the formula 

\begin{equation}
\label{embedformula}
 \tilde{\rho}\Big ( \sum_{i_0,..,i_d\geq 0}{ a_{i_0,..,i_d} \prod_{k=0}^d{v_k^{i_k}}}\Big ) = \max_{i_0,..,i_d\geq 0} {|a_{i_0,..,i_d}| \prod_{k=0}^d{\rho(v_k)^{i_k}}}
\end{equation}
It can be proved (see \cite{rtw}) that $\tilde{\rho}$ does not depend on the choice of the basis $v_0,..v_d$. We define $j([\rho])=[\tilde{\rho}]$.

We summarize some of the properties of $j$ and $\tau$ in the following lemma:
\begin{lemma} \begin{enumerate}
\item
$j$ is a section of $\tau$, i.e. $\tau \circ j = id$.
\item
$\tau$ is surjective and $j$ is injective. 
\item
$\tau$ and $j$ are continuous and $G$-equivariant. 

\end{enumerate} \end{lemma}

\begin{proof}
The first statement is obvious, and the second one follows from it. $\tau$ is obviously $G$-equivariant, and by the first statement, so is $j$. 
Let us prove the continuity of $j$ and $\tau$. The sets $\{[\rho]\,|\,\rho(v) < c \rho(w)\}$ for $c>0$ and $v,w \in V$ are a basis of the topology of $B^d_k$. Similarly, the sets $\{[\rho] \, | \, \rho(f) < c \rho(g)\}$ for $c>0$, $f,g \in K[T_0,..,T_d]$ are a basis of the topology of $\Omega^d_{k,K}$. The continuity of $\tau$ follows immediately. The continuity of $j$ can be easily verified on each apartment of $B^d_{k}$, using fomula (\ref{embedformula}).
\end{proof} 

% XXXXXXXXXXXXXXXXXXXXXXXXXXXXXXXXXXXXX Products of Drindeld spaces XXXXXXXXXXXXXXXXXXXXXXXXXXXXXXXXXXXX

\section{Products of Drindeld spaces}
We now introduce our main objects of inquiry, which are products of Drinfeld spaces over various fields.
As in the introduction, we let $r$ and $d_1,...,d_r$ be natural numbers,  $l$ a non-Archimedean local field, $K$ a non-Archimedean field containing $l$, and $k_1,..,k_r$ finite extensions of $l$, which are contained in $K$.
The following notation will be valid throughout the rest of this paper:

\begin{itemize}

\item
$\X_i=\Omega_{k_i,K}^{d_i}$. $\X_i$ has projective coordinates $T_{i,0},..,T_{i,d_i}$ coming from the embedding $\X_i \hookrightarrow \mathbb{P}^{d_i}_K$, and affine coordinates $t_{i,j}=T_{i,j}/T_{i,0}$ coming from the embedding $\X_i \hookrightarrow \A^{d_i}_K \cong \mathbb{P}^{d_i}_K \setminus \{ z: T_{i,0}(z) = 0 \}$, as explained in section \ref{sec:DrSp}. For convenience, we put $t_{i,0}=1$.
\item
$G_i=PGL_{d_i+1}(k_i)$
\item
$\B_i=B^{d_i}_k$, the Bruhat-Tits building of  $SL_{d_i+1}/k_i$. It is identified with the space of norms on the $k_i$-vector space $V_i:=\bigoplus_j{k_i T_{i,j}}$, modulo homothety. 
\item
$\Lambda_i$ = The basic apartment of $\B_i$, determined by the coordinates $T_{i,0},..,T_{i,d_i}$. Each point of $\Lambda_i$ is represented by a norm of the form $\rho(\sum_j a_jT_{i,j})=\max_j(r_j |a_j|)$, for some positive real numbers $r_0,..,r_{d_i}$.
\item
$\pi_i$ = a uniformizer of $k_i$
\item
$\Delta_i$ = The basic chamber of the apartment $\Lambda_i$, given by the inequalities $|T_{i,0}(x)|\geq |T_{i,0}(x)| \geq .. \geq |T_{i,d_i}(x)| \geq |\pi_i||T_{i,0}(x)|$.
\item
$\X=\prod \X_i$, a product in the category of Berkovich spaces over $K$. We will give a concrete description of $\X$ in terms of seminorms below.
\item
 $G=\prod G_i$, $\B=\prod \B_i$, $\Lambda = \prod \Lambda_i$ and $\Delta=\prod \Delta_i$. $\B$ and $\Lambda$ are polysimplicial complexes, and $\Delta$ is a polysimplex.
\item
For a local field $h$ contained in $K$, we put $\B_i(h)=B^{d_i}_h$ and $\B(h)=\prod \B_i(h)$.
\item
$k=$ the compositum (in $K$) of all the fields $k_i$. Then $k$ is a complete local field and the extensions $k/k_i$ are finite.  
\end{itemize}

Let us describe the points of $\X$ in terms of seminorms. The product $\prod_i \PR^{d_i}_K$ is the Berkovich space whose points are multiplicative seminorms $\rho$ on $K[T_{i,j}]$ (for $1\leq i\leq r$, $0 \leq j \leq d_i$) which extend the norm on $K$, such that for each $1\leq i\leq r$, the numbers $\rho(T_{i,0}),..,\rho(T_{i,d_i})$ are not all zero, modulo the following equivalence relation: $\rho_1 \sim \rho_2$ if and only if there exist constants $c_1,..,c_r$ such that for any polynomial $p \in K[T_{i,j}]$ which is homogenous in each group of variables $T_{i,j}$ (for $0\leq j \leq d_i$) of degree $e_i$, $$\rho_1(p)=(\prod_{i=1}^r c_i^{e_i}) \rho_2(p)$$
This can be proved, as in the case of a single projective space, using the affine covering by the affine sets defined by the conditions $T_{i,j_i}(x) \neq 0 \text{ for } 1\leq i \leq r$,  for all possible choices of indices $0 \leq j_i \leq d_i$. 

We now define

$$ \X = \Big \{[\rho] \in \prod_i \PR^{d_i}_K \, | \,  \rho(\sum_j \alpha_j T_{i,j}) \neq 0 \text{ for all }i \text{ and } (\alpha_0,..,\alpha_{d_i})\in k_i^{d_i+1}\setminus \{0\} \Big \} $$

We cover $\X$, as we covered a single Drinfeld space in section \ref{sec:DrSp}, by the following subsets:

\begin{equation*}
\begin{split}
\X(n)=\Big \{   [\rho]\in \prod_i \PR^{d_i}_K \, | \, & \rho(\sum_{j=0}^{d_i}\alpha_j T_{i,j}) > |\pi_i|^n   \max_j {|\alpha_j|}\max_j {\rho(T_{i,j})}  \\
 & \text{ for all } i \text{ and } (\alpha_0,..,\alpha_{d_i})\in {k_i}^{d_i+1} \setminus \{0\}\Big \} 
\end{split}
\end{equation*}

and 

\begin{equation*}
\begin{split}
\X[n]=\Big \{   [\rho]\in \prod_i \PR^{d_i}_K \, | \, & \rho(\sum_{j=0}^{d_i}\alpha_j T_{i,j}) \geq |\pi_i|^n   \max_j {|\alpha_j|}\max_j {\rho(T_{i,j})}  \\
 & \text{ for all } i \text{ and } (\alpha_0,..,\alpha_{d_i})\in {k_i}^{d_i+1} \setminus \{0\} \Big \} 
\end{split}
\end{equation*}
\begin{prop}
\label{prodstruct}
\begin{enumerate}

\item For all $n$, $\X[n]$ is a strictly affinoid domain in $\prod_i \PR^{d_i}_K$.
\item For all $n$, $\X(n)$ is an open subset of $\prod_i \PR^{d_i}_K$.
\item $\X = \bigcup_{n \geq 1} \X(n) =  \bigcup_{n \geq 1} \X[n]$.
\item $\X$ is an open subset of $\prod_i \PR^{d_i}_K$, and therefore has a Berkovich space structure.
\item $\X \cong \prod_i \X_i$.

\end{enumerate}
\end{prop}

\begin{proof}
The first four claims follow exactly as the claims in proposition \ref{open}. Let us prove the last one. Let us cover each $\X_i$ by open sets $\X_i(n)$ and by affinoids $\X_i[n]$ as in the proof of proposition \ref{open}. Each $\X_i[n]$ is the Berkovich spectrum of an affinoid algebra (see \cite{ber2}, remark 2.2.2(i) for the generators and relations), and the complete tensor product of these algebras is the affinoid algebra which defines $\X[n]$. Hence, $\X[n] =\prod_i \X_i[n]$.
Let us show that $\X \cong \prod_i \X_i$ by verifying the universal property of the product. There are maps $\pi_{\X_i}:\X \rightarrow \X_i$, coming from the projection maps from $\prod_i \PR^{d_i}_K$ to its factors. Given a $K$-affinoid space $Y$ and maps $f_i:Y \rightarrow \X_i$, since $Y$ is compact, $f_i(Y)$ is compact for all $i$. Since the sets $\X_i(n)$ are open and cover $\X$, there exists $n$ such that $f_i(Y) \subseteq \X_i(n)$ for all $i$. Hence, there is a unique map $f: Y\rightarrow \X[n]$ which commutes with the projection maps from $\X[n]$ to all the $\X_i[n]$'s. Composing $f$ with the inclusion of $\X[n]$ in $\X$, we get a map from $Y$ to $\X$ which commutes with the maps $\pi_{\X_i}$. Its uniqueness follows from the uniqueness of $f$.  
\end{proof}

We now discuss the relation between the building $\B=\prod_i \B_i$ and the space $\X= \prod_i \X_i$. We have seen in section \ref{sec:btb} that each $\B_i$ is embedded, as a topological space, in $\X_i$ via a $G_i$-equivariant embedding $j_i$, and that there exists a projection $\tau_i:\X_i \rightarrow \B_i$ which is also $G_i$-equivariant. We have similar maps
\[ \begin{array}{c}
\X \\
j\uparrow \downarrow \tau \\
\B
\end{array}\]
which are continuous and $G$ - equivariant, and satisfy $\tau \circ j = id$. The map $\tau$ is defined via the $\tau_i$'s by 
$$\tau(x)=(\tau_1(\pi_{\X_1}(x)), .., \tau_r(\pi_{\X_r}(x)))$$ where $\pi_{\X_i}$ is the projection from $\X$ to $\X_i$. The map $j$ is defined as follows: Given $r$ equivalence classes of norms $[\rho_i]\in \B_i$, we take a basis $e_{i,0},..,e_{i,d_i}$ of each $V_i$, such that $[\rho_i]$ lies in the apartment corresponding to this basis. We define $j(([\rho_1],..,[\rho_r]))=[\rho]$ where $\rho$ is the  seminorm on $K[T_{i,j}]=K[e_{i,j}]$ defined by 

\begin{equation}
\label{embedeq}
\rho \Big(\sum_{N=(n_{i,j})} a_N \prod_{i,j}e_{i,j}^{n_{i,j}}\Big)=\max_N |a_N| \prod_{i,j}\rho(e_{i,j})^{n_{i,j}}
\end{equation}
For convenience, we will identify $\B$ with its image in $\X$.

We note that the $(B,N)$-pair structure on each $G_i$ give rise to a $(B,N)$-pair structure on $G$, and $\B$ is the building associated to this structure. Hence, the results of the second chapter of \cite{brt} may be applied to it. In particular, we have the following results:

\begin{lemma}\label{citebrt}
\begin{enumerate} 
\item The group $G$ acts transitively on the set of apartments of $\B$.
\item The stabilizer in $G$ of an apartment of $\B$ acts transitively on the set of chambers of the apartment.
\item Any polysimplicial automorphism of $\B$ carries apartments to apartments.
\end{enumerate}
\end{lemma}
\begin{proof} See \cite{brt}, corollaries 2.2.6 and 2.8.6. \end{proof}

% XXXXXXXXXXXXXXXXXXXXX Connectedness and a uniqueness principle XXXXXXXXXXXXXXXXXXXXXXX

\section{ Connectedness and a uniqueness principle }
The following result will be used in the proof of our main theorem.
\begin{prop} \label{prop:uniq}
Let $f$ be an analytic function on $\X$, and let $Y$ be an open subset of $\X$. If $f|_Y=0$ then $f=0$.
\end{prop}

Recall that a Berkovich space $X$ is called regular if the local rings $O_{X,x}$ are regular, and that for an affinoid algebra $\AF$, $M(\AF)$ is regular if and only if $\AF$ is a regular ring (see \cite{ber2}, proposition 2.3.4).

We will prove proposition \ref{prop:uniq} by showing that the covering affinoids $\X[n]$ are connected and regular, and that for such affinoids, the canonical map from the affinoid algebra to the local ring at any point is injective.

\begin{lemma}
For each $n\geq 1$, $\X[n]$ is pathwise connected and regular.
\end{lemma}

\begin{proof}
Let $d$ be any natural number. Let us recall from \cite{ber2} a certain deformation retraction from $\A^d_K$ to a $d$-dimensional Euclidean space.
Let $R=K[x_1,..,x_d]$. For a multi-index $N=(n_1,..,n_d)$ of nonnegative integers, Let $D_N$ be the following operator on $R$:
$$ D_N \Big( \sum_{I=(i_1,..,i_d)\geq 0} a_I x^I  \Big)= \sum_{I \geq N} \binom{i_1}{n_1} \cdot..\cdot \binom{i_d}{n_d} a_I x^I $$
Let $\rho \in \A^d_K$. $\rho$ is a multiplicative seminorm on $R$ which extends the absolute value on $K$. Let us define, for any real number $0 \leq t \leq 1$, 
\begin{equation}
\label{rhodef}
 \rho_t (p) = \max _{N=(n_1,..,n_d)\geq 0} t^{|N|} \rho(D_N (p))
\end{equation}
where $|N|=n_1 + .. + n_d$. It is shown in \cite{ber2}, corollary 6.1.2 and remark 6.1.3(ii), that $\rho_t \in \A^d_K$ for all $t$, and the map $t \mapsto \rho_t$ is continuous.
Clearly, $\rho_0=\rho$. $\rho_1$ is given by
\begin{equation}
\label{rho_1}
\rho_1 \Big( \sum_{N=(n_1,..,n_d)} a_N x^N \Big)= \max_N |a_N| \prod_j \rho(x_j)^{n_j}
\end{equation}
since for $t=1$, the maximum in equation (\ref{rhodef}) is attained at a multi-index $N_0$ for which $\rho(a_{N_0} x^{N_0})=\max_N \rho(a_{N} x^{N})$ and $|N_0|$ is maximal. 

Now, let $d=\sum_i (d_i+1)$ and let $x=[\rho]$ be a point of $\X$ such that $\tau(x) \in \Lambda$. $\rho$ is a seminorm on $K[T_{i,j}]$. 
Let us consider the points $x(t)=[\rho_t]$ for $0\leq t \leq 1$. By  (\ref{embedeq}) and (\ref{rho_1}), $x(1)=\tau(x)$. Hence, the points $\{x(t) \, | \, 0\leq t \leq 1 \}$ form a path from $x$ to $\tau(x)$. Moreover, for all $i$, if $p=\sum_{j} a_{i,j}T_{i,j}$ is any element of $V_i$,
then for any multi-index $N$,  $D_N (p)=\sum_{j\in J} a_{i,j}T_{i,j}$, where $J=\{j \,|\, n_{i,j} \leq 1 \}$. Hence, by the assumption that $\tau(x) \in \Lambda$, we have

$$\rho(D_N (p))=\max_{j \in J} \rho(a_{i,j}T_{i,j}) \leq \max_{0\leq j \leq d_i} \rho(a_{i,j}T_{i,j}) = \rho(p) $$

We conclude that $\rho_t(p)=\rho(p)$. Since $\rho$ and $\rho(t)$ are identical on $V_i$ for all $i$, $\tau(x)=\tau(x(t))$.   
Therefore, the path $\{x(t) \, | \, 0\leq t \leq 1 \}$ lies inside the fiber $\tau^{-1}(\tau(x))$. 
Let $y$ be a point of $\X(n)$, and let $g\in G$ be such that $g \tau(y) \in \Lambda$. Let $x=gy$ and $y(t)=g^{-1}x(t)$ for $0\leq t \leq 1$. Since $\tau$ is $G$-equivariant, the points $\{y(t) \, | \, 0\leq t \leq 1 \}$ form a path from $y$ to $\tau(y)$ inside the fiber $\tau^{-1}(\tau(y))$. 
Since $\X[n]$ is defined inside $\prod_i \PR^{d_i}_K$ by conditions of the form $\rho(a)\geq\rho(b)$ for some linear polynomials $a,b \in K[T_{i,j}]$, and $\tau$ is defined by restricting seminorms to linear polynomials, $\X[n]$ is a union of fibers of $\tau$. Hence $\{ y(t) \,| \, 0 \leq t \leq 1\}$ is a path from $y$ to $\tau(y)$ inside $\X[n]$. We conclude that each point of $\X[n]$ can be pathwise connected to a point in $\X[n] \cap \B$. By the above remark on the definition of $\X[n]$,  $\X[n] \cap \B$ is a convex set. Hence it is pathwise connected. We conclude that $\X[n]$ is pathwise connected.

Finally, by proposition \ref{prodstruct}, $\X[n]$ is an affinoid domain inside an affine $K$-analytic space, which is regular, by \cite{ber2}, proposition 3.4.3. By \cite{ber4}, corollary 2.2.8, $\X[n]$ is regular.
\end{proof}

\begin{lemma} \label{inject}
Let $\AF$ be a regular affinoid algebra, such that $X=M(\AF)$ is connected. Then \begin{enumerate}
\item $\AF$ is an integral domain.
\item For any $x \in X$, the canonical map from $\AF$ to $O_{X,x}$ is injective. 
\end{enumerate}
\end{lemma}

\begin{proof}
\begin{enumerate}
\item
$\AF$, as an affinoid algebra, is Noetherian. By \cite{kap}, Theorem 168, $\AF$ is a finite product of integral domains: $\AF=\prod_{i=1}^n A_i$. Let us decompose $1=a_1+..+a_n$ with $a_i \in A_i$. Then, $X$ is a disjoint union of the supports of $a_i$, which are nonempty open sets. If $n>1$, this contradicts the connectedness of $X$. Hence, $n=1$.
\item
Let $P$ be the ideal $\{f\in \AF\,|\, |f(x)|=0\}$. Then $P$ is prime and the map from $\AF$ to $\AF_P$ is injective. By \cite{ber4}, Theorem 2.1.4, the canonical map from $\AF_P$ to $O_{X,x}$ is faithfully flat, and in particular, injective. Therefore, the composition of the two maps is injective. \qedhere
\end{enumerate}
\end{proof}

\begin{proof}[Proof of proposition \ref{prop:uniq}]
Let $x$ be a point of $Y$, and let $N$ be such that $x \in \X(N)$. For any $n \geq N$, the image of $f|_{X[n]}$ in $O_{X[n],x}$ is $0$, and by the above two lemmas, $f|_{X[n]}=0$. Hence $f=0$.
\end{proof}
% XXXXXXXXXXXXXXXXXXXXX  The induced automorphism on the building XXXXXXXXXXXXXXXXXXXXXXXXXX

\section{The induced automorphism on the building}

It is a crucial step in the proof of our main theorem to show that an analytic automorphism of $\X$ preserves $\B$.
We will prove it in this section, along with some combinatorial properties of the induced automorphism.

Recall that each edge of $\B$ comes from an edge of one of the factors $\B_i$.
\begin{definition}
We will say that an edge $e$ of $\B$ which comes from $\B_i$ has length $l(e):=-\log |\pi_i|$. A polysimplicial automorphism of $\B$ will be called \emph{length preserving} if it carries any edge of $\B$ to an edge of the same length.
\end{definition}

\begin{prop}\label{induce} Let $\phi$ be an analytic automorphism of $\X$. Then 
\begin{enumerate}
\item
$\phi(\B)=\B$.
\item The induced automorphism on $\B$ is polysimplicial, chambered, and length preserving.
\item $\phi$ commutes with the projection $\tau:\X \rightarrow \B$.
\end{enumerate}
\end{prop}

\begin{proof}
We have a partial ordering $\leq$ on any analytic space $X$ defined as follows: $x \leq y$ if for any analytic function $f$ on $X$ we have $|f(x)|\leq |f(y)|$. A point $x\in X$ is called \emph{maximal} if $y\geq x$ holds only for $y=x$. $x$ is called a \emph{global maximum} if $x \geq y$ for all $y\in X$. Note that if an affinoid $X=M(\AF)$ has a global maximum, then it is unique.

We will show that $\B$, as a subset of $\X$, is characterized as the set of maximal points of $\X$, and thus, is preserved under any analytic automorphism of $\X$. We start with a few lemmas:

\begin{lemma}\label{unimax} Let $X=M(\AF)$ be an affinoid domain over $K$ which has a point $x$ which is a global maximum. Let $f$ be an element of $\AF$ such that $r:=|f(x)|>0$. Let $Y \subseteq X$ be the Laurent domain $\{y\in X : |f(y)|=r\}$. Then $Y$ is an affinoid domain in which $x$ is a global maximum. \end{lemma}

\begin{proof} $Y$, and its embedding in $X$, are given by $Y=M(B)$ where $B=\AF\{r^{-1}X,rY\}/\langle X-f,Yf-1 \rangle$ (see  \cite{ber2}, 2.2.2, and \cite{bgr}, 6.1.4). Since $|f|_{\sup} \geq |f(x)| =r$, by \cite{ber2}, corollary 2.1.5, we can construct an isomorphism $B\cong \AF \{rY\}/\langle Yf-1 \rangle$. In other words, $Y= \{ y\in X : |f(y)|\geq r\}$. By \cite{ber3}, lemma 5.1(i), the restriction of the partial order $\leq$ on $X$ to $Y$ coincides with the partial order on $Y$. Since $x\in Y$, $x$ is also the global maximum of $Y$. \end{proof}

\begin{lemma} \begin{enumerate}
\item
For all $x\in \B$, $\tau(x) \geq x$.
\item
If $x,y\in \X$ and $\tau(x) \neq \tau(y)$ then none of the relations $x \leq y$, $y \leq x$ hold.
\end{enumerate}
\end{lemma}

\begin{proof}
\begin{enumerate}
\item
Let  $y=\tau(x)$. Since $G$ acts transitively on the set of apartments of $\B$, we can assume, without loss of generality, that $y$ lies in the basic apartment $\Lambda$. The preimage of $y$ under $\tau$ is the set of $z\in \X$ such that for all $i$, and for all $a_0,..,a_{d_i}\in k_i$, we have $|\sum_{j=0}^{d_i} a_j t_{i,j}(z)|=|\sum_{j=0}^{d_i} a_j t_{i,j}(y)|$. Also, since $y\in \Lambda$, we have $|\sum_j a_j t_{i,j}(y)|=\max_j |a_j t_{i,j}(y)|$ for all such $a_j$'s. Moreover, if $z \in \prod_i  \A^{d_i}_K$ and $|\sum_{j=0}^{d_i} a_j t_{i,j}(z)|=\max_j |a_j t_{i,j}(y)|$ holds for all $i$ and $a_0,..,a_{d_i}\in k_i$, then $z \in \X$. All together, we get 
\begin{equation}
\label{taueq}
\tau^{-1}(y)= \{z \in \prod_i  \A^{d_i}_K \, | \, |\sum_j a_j t_{i,j}(z)|=\max_j |a_j t_{i,j}(y)| \text{ } \forall i,  a_0,..,a_{d_i}\in k_i\}
\end{equation}
This set of conditions may be replaced by a finite set of conditions in the following way: For every $i$ and a real number $|\pi_i| < b \leq 1$, let 

$$J_{i,b}= \{j : |t_{i,j}(y)| \equiv b \mod |\pi_i|^{\Z} \}$$

Obviously, this set is nonempty only for a finite number of $b$'s. For such $b$, let $n_{i,j}$ be so that $|t_{i,j}(y) \pi_i^{n_{i,j}}| = b$. Also, let $C_i\subset O_{k_i}$ be a finite set of representatives for the residue classes mod $\pi_i$, such that $0\in C_i$. Then our condition is that for all $i$,$b$, and all possible functions $f:J_{i,b} \rightarrow C_i$ which are not identically 0, we have $|\sum_{j\in J_{i,b}} f(j) t_{i,j}(z) \pi_i^{n_{i,j}}|=b$. This finite set of conditions is equivalent to the infinite set of conditions in (\ref{taueq}).

Let $U$ be the polydisc
$$ U=\{z \in \prod_i \A^{d_i}_K \, | \, \forall i,j,\,|t_{i,j}(z)| \leq |t_{i,j}(y)| \} = M(K\{|t_{i,j}(y)|^{-1}t_{i,j}\}) $$
then $\tau^{-1}(y)$ is the Laurent domain inside $U$, defined by the finite set of conditions described above. 
The image of $y$ in $U$ is the Gauss norm $$g(\sum_{N=(n_{i,j})} a_N \prod_{i,j} t_{i,j}^{n_{i,j}})= \max_N |a_N| \prod_{i,j} |t_{i,j}(y)|^{n_{i,j}}$$ which is the global maximum of $U$. Applying lemma \ref{unimax} successively, we get that $y$ is the global maximum of $\tau^{-1}(y)$. In particular, we have $y \geq x$.

\item
The claim follows from the case $r=1$, proved in \cite{ber}. If $\tau(x) \neq \tau(y)$ then there is an index $i$ such that $\pi_{\B_i}(\tau(x)) \neq \pi_{\B_i}(\tau(y))$. By the case $r=1$, there exist analytic functions $f,g$ on $\X_i$ such that $|f(\pi_{\X_i}(x))| > |g(\pi_{\X_i}(x))|$ and  $|f(\pi_{\X_i}(y))| < |g(\pi_{\X_i}(y))|$. Pulling $f$ and $g$ back to $\X$, we conclude that there is no order relation between $x$ and $y$. \qedhere
\end{enumerate}
\end{proof}

By the above lemma, $\phi(\B)=\B$, and the induced automorphism on $\mathcal{B}$ is continuous and commutes with $\tau$. Let us show that this automorphism is polysimplicial. We start with a lemma:

\begin{lemma} \label{affine}
Let $f$ be a nowhere vanishing analytic function on $\tau^{-1}(\Delta^\circ)$. Then $\log |f|$, when restricted to $\Delta^\circ$, is an affine function, whose linear part has integer coefficients. 
\end{lemma} \begin{proof}
$\tau^{-1}(\Delta_i^\circ)$ is the analytic subspace of $\prod_i  \A^{d_i}_K$ given by the inequalities $ 1 > |t_{i,1}(x)| > |t_{i,2}(x)| > ... > |t_{i,d_i}(x)| > |\pi_i|$ for $1 \leq i \leq r$.
There is a cover of $\Delta^{\circ}$ by sets of the form $\{ x\in \Delta^\circ: \forall i,j,\text{ } a_{i,j} \leq |t_{i,j}(x)| \leq b_{i,j}\}$ (for some real numbers $a_{i,j},b_{i,j}$) and a corresponding cover of $\tau^{-1}(\Delta_i^\circ)$ by multi-annuli. 
On each such multi-annulus, an invertible function is of the form $\alpha \prod_{i,j}t_{i,j}^{n_{i,j}}(1+h)$ where $n_{i,j}\in \Z$ and $|h|<1$ (see \cite{bgr}, 9.7.1). Taking the absolute value, the result follows.
\end{proof}

 Now, let us show that $\phi$ takes chamber interiors to chamber interiors: Assume, to the contrary, that $x$ is a point in the interior of a chamber $C$, and $y=\phi(x)$ lies in an $n$-dimensional face $F$, where $n$ is not the top dimension $d=\prod_i d_i$. Then (by slightly moving the point $x$ if necessary) we can assume that $n=d-1$. It follows from the above lemma that $\phi$ is linear on $C \cap \phi^{-1}(F)$. Hence, there exists a small neighbourhood of $x$ in $C$ such that $\phi^{-1}(F)$ is defined by a linear equation inside this neighbourhood. We may take $U$ to be a small ball around $x$ in $C$, so that $U \setminus \phi^{-1}(F)$ has two connected components. However, $\phi(U)\setminus F$ has more than two connected components, since $F$ connects the chambers which contain it. We get a contradiction. 

Now, since $\phi$ takes chamber interiors to chamber interiors, it takes chambers to chambers. Since $\phi$ is affine on each chamber, and vertices of a chamber are characterized as the only points who do not lie on the interior of an interval which lies in the chamber, $\phi$ takes chamber vertices to chamber vertices. Each chamber is a product of simplices, and two vertices of a chamber are connected by an edge if and only if they come from an edge of one of the simplices of the product. Equivalently, two vertices of a chamber are connected by an edge if and only if the segment connecting them does not contain interior points of the chamber. It follows that $\phi$ takes edges to edges. From here, and lemma \ref{graphaut} below, it follows that $\phi$ is polysimplicial. 

Finally, let us show that $\phi$ is length preserving: let $e$ be an edge of $\B$, and let $f=\phi(e)$. Since $G$ acts on $\B$ by length preserving automorphisms, and by lemma \ref{citebrt}, we may assume without loss of generality that both $e$ and $f$ lie in the apartment $\Lambda$. Let $x_1$ and $x_2$ be the end points of $e$, and $y_n=\phi(x_n) \text{ } (n=1,2)$ the end points of $f$. The edge $e$ comes from an edge in one of the $\Lambda_i$'s. Let us assume that it comes from $\Lambda_{i_0}$. Then $\pi_{\B_{i_0}}(x_1)$ and $\pi_{\B_{i_0}}(x_2)$ correspond to the homothety classes of two lattices $L_1, L_2$ of the form $L_n=\langle \pi_{i_0}^{m_{n,0}}T_{i_0,0}, ..,\pi_i^{m_{n,d_i}}T_{i_0,d_i}\rangle$, such that $\max_j(m_{1,j}-m_{2,j})-\min_j(m_{1,j}-m_{2,j}) = 1$. By formula (\ref{norm_formula}), $\rho_{L_i}(T_{i,j})=|\pi_{i_0}|^{m_{n,j}}$. 
It follows that
\begin{equation*}
\begin{split}
l(e) = - \log(|\pi_{i_0}|)& = \log \Big (\Big(\min_{0 \leq j \leq d_{i_0}} \frac{\rho_{L_1}(T_{i_0,j})}{\rho_{L_2}(T_{i_0,j})} \Big) /  \Big(\max_{0 \leq j \leq d_{i_0}} \frac{\rho_{L_1}(T_{i_0,j})}{\rho_{L_2}(T_{i_0,j})} \Big) \Big )\\
& =\sum_i \log \Big(\Big(\min_{0 \leq j \leq d_i} \frac{|t_{i,j}(x_1)|}{|t_{i,j}(x_2)|}\Big) / \Big (\max_{0 \leq j \leq d_i} \frac{|t_{i,j}(x_1)|}{|t_{i,j}(x_2)|}\Big)\Big) 
\end{split}
\end{equation*}
Hence $l(e)$ is an integral combination of the numbers $(\log |t_{i,j}(x_1)| - \log |t_{i,j}(x_2)|)$. 
Let $v_n$ be the vector $(\log |t_{i,j}(x_n)|)_{i,j}$ and $w_n$ the vector $(\log |t_{i,j}(y_n)|)_{i,j}$ for $n=1,2$.  According to lemma \ref{affine}, there exists a matrix $A$ with integral coefficients and a vector $b$ such that $w_n=Av_n+b$ for $n=1,2$. Hence, $w_1-w_2= A(v_1-v_2)$. Since the vector $v_1-v_2$ contains only $0$ and $l(e)$ as entries, all its entries are integer multiples of $l(e)$. Hence, the same is true for the entries of $w_1-w_2$. Since $w_1-w_2$ is not the zero vector, we get that $l(f)\geq l(e)$. Applying the same argument to $\phi^{-1}$, we get $l(e)\geq l(f)$. Hence $\phi$ is length preserving.

This concludes the proof of proposition \ref{induce}. \end{proof}

% ************************** The action on a single chamber ***********************

\section{The action on a single chamber}
\label{labelling}
We have just seen that an analytic automorphism $\phi\in \Aut_K(\X)$ induces an automorphism of $\B$ with some good properties. We will now explore this polysimplicial automorphism in more detail. By lemma \ref{citebrt}, the image $\phi(\Lambda)$ is an apartment. Also, $G$ acts transitively on the set of apartments of $\B$, and the stabilizer of $\Lambda$ acts transitively on the set of chambers of $\Lambda$. Thus, in proving theorem \ref{T}, it is enough to consider automorphisms $\phi$ which satisfy $\phi(\Lambda) = \Lambda$ and $\phi(\Delta)=\Delta$.

Below, we shall analyze the action of such $\phi$ on $\Delta$. Let us consider the 1-skeleton of $\Delta$. It endows the vertices of $\Delta$ with an undirected graph structure, which is the product of the graph structures on each of the simplices $\Delta_i$: Two vertices  $(\alpha_1,..,\alpha_r)$ and $(\beta_1,..,\beta_r)$ (where $\alpha_i, \beta_i$ are vertices of $\Delta_i$) are connected by an edge if an only if $\alpha_i \neq \beta_i$ for exactly one value of $i$. Thus, the graph of $\Delta$ is isomorphic to $\prod_{i=1}^r [d_i]$, where $[n]$ denotes the complete graph on the set $\{1,2,..,n\}$. Since $\phi$ is polysimplicial, it induces an automorphism of this graph. The automorphism group of this graph is given by the following combinatorial lemma:
\begin{lemma}\label{graphaut} 
\begin{enumerate} \item Let $n_1,n_2,a_1,..,a_{n_1},b_1,..,b_{n_2}$ be natural numbers, and let $f$ be an injective homomorphism of graphs from $\prod_{i=1}^{n_1}[a_i]$ to $\prod_{i=1}^{n_2}[b_i]$. Then, there exists an injective function $\mu: [n_1] \rightarrow [n_2]$, injective functions $g_i:[a_i]\rightarrow[b_{\mu(i)}]$ for all $i\in [n_1]$, and a number $\alpha_j\in [b_j]$ for each $j \in [n_2] \setminus \text{Im}(\mu)$, such that for any $(u_1,..,u_{n_1}) \in \prod_{i=1}^{n_1}[a_i]$ and $j\in [n_2]$, the $j^\text{th}$ coordinate of $f(u_1,..,u_{n_1})$ is $g_i(u_i)$ if $j=\mu(i)$ for some $i$, and is $\alpha_j$ if $j \notin \text{Im}(\mu)$.
\item
Let $f$ be an automorphism of the graph $\prod_{i=1}^r [d_i]$. Then there exists a permutation $\mu$ of $\{1,..,r\}$ such that $d_{\mu(i)}=d_i$ for all $i$, and permutations $p_i$ of $[d_i]$ for all $i$ such that $f((a_1,..,a_r))=(p_1(a_{\mu(1)}),..,p_r(a_{\mu(r)}))$.\end{enumerate}\end{lemma}
\begin{proof} For the first claim, see \cite{ber3}, lemma 3.1. The second claim follows immediately from the first one. \end{proof}

We conclude from the second part of the lemma that the action of $\phi$ on $\Delta$ may be factored as an action on each $\Delta_i$ by a permutation, and a permutation of the factors $\Delta_i$. It would make the proof of theorem \ref{T} easier if we could assume (after modifying $\phi$ by a suitable element of $G$) that the action on each $\Delta_i$ is trivial. However, not every automorphism of $\Delta_i$ is induced by an element of $G_i$. In fact, some automorphisms of $\Delta_i$ are not even induced by a polysimplicial automorphism of $\B_i$. As we shall see, these questions are best handled by analyzing the interaction between $\phi$ and the \emph{labelling} of the building. 

\begin{definition} A labelling of a polysimplicial complex $A$ is a map from the set of vertices of $A$ to the set of vertices of a single closed polysimplex $F$, which extends to a polysimplicial chambered map from $A$ to $F$. 
\end{definition}

\begin{definition} \label{def:label}
\begin{enumerate}
\item
For any $1\leq i \leq r$, let $C_i$ be the labelling of $\B_i$, with values in the set of residues $\Z / d_i \Z$ (which is regarded as a combinatorial simplex), defined in the following way: let $v_i$ be the normalized discrete valuation on $k_i^*$ (so that $v_i(k_i^*)=\Z$), then the label of a vertex $[L]\in \B_i$ is the residue class of $v_i(\det(u_1,..,u_{d_i}))$, where $(u_1,..,u_{d_i})$ is an $O_{k_i}$-basis of $L$.
\item
Let $C$ be the labelling of $\B$ which is the product of the labellings $C_i$, with values in $\prod_i(\Z/d_i\Z)$.
\end{enumerate}
\end{definition}

It is easy to see that $C_i$ and $C$ are indeed labellings.

The next lemma, on the uniqueness of labelling, is a polysimplicial version of the proposition in \cite{gar}, 4.4.

\begin{lemma}
\label{lem:label}
Let $C_1,C_2$ be two labellings of a polysimplicial complex $A$, and let $\Phi$ be a chamber of $A$ such that $C_1|_\Phi = C_2|_\Phi$. Then $C_1=C_2$.
\end{lemma}

\begin{proof}
Since any chamber of $A$ can be connected to $\Delta$ by a gallery, it is enough to prove that if $\Gamma_1$ and $\Gamma_2$ are adjacent chambers, and $C_1=C_2$ on $\Gamma_1$, then $C_1=C_2$ on $\Gamma_2$. This follows from the following claim: Given integers $l\geq 1$ and $u_1,..,u_l>1$, if $f$ is an automorphism of $\prod_{i=1}^l [u_i]$ such that $f$ is the identity on the set $[u_1]\times ... \times \underset{i_0}{([u_{i_0}] \setminus \{j \})} \times ... \times [u_l]$ for some $1\leq i_0 \leq l$ and $1\leq j \leq u_{i_0}$, then $f$ is the identity on $\prod_{i=1}^l [u_i]$. This claim follows from part 2 of lemma \ref{graphaut}.
\end{proof}

In a similar way, we prove a polysimplicial version of the uniqueness lemma in \cite{gar}, 3.2:

\begin{lemma} \label{aptaut}
Let $A$ and $B$ be two polysimplicial complexes, such that every facet of $B$ is contained in at most two chambers. Let $\Psi$ be a chamber of $A$ and $f,g:A\rightarrow B$ chambered, injective polysimplicial maps such that $f|_\Psi = g|_\Psi$. Then $f=g$.
\end{lemma}

\begin{proof}
As in the proof of lemma \ref{lem:label}, it is enough to prove that if $\Gamma_1$ and $\Gamma_2$ are adjacent chambers, and $f=g$ on $\Gamma_1$, then $f=g$ on $\Gamma_2$. Let $\Phi_i=f(\Gamma_i)$ for $i=1,2$. Then $\Phi_1$ and $\Phi_2$ are adjacent, and $g(\Gamma_1)=\Phi_1$. Let $F=\Gamma_1 \cap \Gamma_2$ and  $G=\Phi_1 \cap \Phi_2$. Then $F$ and $G$ are facets and $g(\Gamma_1)\cap g(\Gamma_2)=g(F)=f(F)=G$. Hence, by the assumption on $B$, $g(\Gamma_2)=\Phi_2=f(\Gamma_2)$. We have $f=g$ on $F$, hence $fg^{-1}$ is an automorphism of $\Gamma_2$ which is the identity on $F$. By the claim stated in the proof of lemma \ref{lem:label}, $fg^{-1}$ is the identity on $\Gamma_2$. Hence $f=g$ on $\Gamma_2$.
\end{proof}

The labelling $C$ (see definition \ref{def:label}) induces a bijection between the vertices of $\Delta$ and the set $\prod_i(\Z/d_i\Z)$. Let $D: \prod_i(\Z/d_i\Z) \rightarrow \Delta$ be the inverse mapping to $C|_{\Delta}$.  

\begin{lemma} \label{lem:color}
Let $\phi$ be a polysimplicial automorphism of $\B$ such that $\phi(\Delta)=\Delta$. Then, 
\begin{enumerate}
\item
$C \circ \phi = C \circ \phi \circ D \circ C$.
\item
There exist a unique permutation $\mu$ of $\{1,2,..,r\}$ and unique permutations $p_i$ of $\Z/d_i\Z$ for $1\leq i \leq r$, such that $d_{\mu(i)}=d_i$  for all $i$, and $(C \circ \phi \circ D)((a_1,..,a_r))=(p_1(a_{\mu(1)}),p_2(a_{\mu(2)}),..,p_r(a_{\mu(r)}))$ for any $(a_1,..,a_r) \in \prod_i(\Z/d_i\Z)$.
\item
Each permutation $p_i$ is either of the form $t \mapsto a_i+t$ or $t \mapsto a_i-t$ for some constant $a_i\in \Z/d_i\Z$.
\end{enumerate}
\end{lemma}
\begin{proof}

Since $C \circ \phi$ and $C \circ \phi \circ D \circ C$ are two labellings which are equal on $\Delta$, the first statement follows from lemma \ref{lem:label}. Since $C\circ \phi \circ D$ is an automorphism of $\prod_i(\Z/d_i\Z)$, by lemma \ref{graphaut}, there exist a permutation $\mu$ of $\{1,2,..,r\}$ and permutations $p_i$ of $\Z/d_i\Z$ for $1\leq i \leq r$, such that $d_{\mu(i)}=d_i$ for all $i$ and $(C \circ \phi \circ D)((a_1,..,a_r))=(p_1(a_{\mu(1)}),p_2(a_{\mu(2)}),..,p_r(a_{\mu(r)}))$ for any $(a_1,..,a_r) \in \prod_i(\Z/d_i\Z)$. This proves the second statement.

For the last statement, let $v$ be a vertex of $\B$, of label $(a_1,..,a_r)$, let $1\leq j \leq d_i$ and let $i=\mu(j)$. Let $A$ be the set of neighbours of $v$ (in the 1-skeleton of the building) whose label is $(a_1,..,a_{i-1},t,a_{i+1},..,a_r)$ for some $t$. For any $w\in A$, we have $C(w)=(a_1,..,a_{i-1},t,a_{i+1},..,a_r)$, hence \begin{eqnarray*} 
C(\phi(w))&=&C(\phi(D(C(w))))=C(\phi(D((a_1,..,a_{i-1},t,a_{i+1},..,a_r))))\\ &=&(p_1(a_{\mu(1)}),..,p_j(t),..,p_r(a_{\mu(r)})),\end{eqnarray*} 

Let $b_k=p_k(a_{\mu(k)})$ for all $k$, and let $B$ be the set of neighbours of $\phi(v)$ whose label is of the form $(b_1,..,b_{j-1},s,b_{j+1},..,b_r)$ for some $s$. We conclude that $\phi$ maps $A$ to $B$ bijectively. A vertex of $A$ labelled $(a_1,..,a_{i-1},t,a_{i+1},..,a_r)$ is mapped to a vertex of $B$ labelled $(b_1,..,b_{j-1},p_j(t),b_{j+1},..,b_r)$. 
Let us count how many vertices in $A$ are of each label: The number of vertices of $A$ of label $(a_1,..,a_i+w,..,a_r)$ is equal to the number of sublattices $L'$ of a lattice $L$ of $O_{k_i}^{d_i}$ which contain $\pi_i L$ and such that the length of the module $L/L'$ over $O_{k_i}$ is $w$. This number is equal to the number of  $\tilde{k_i}$- vector subspaces of $\tilde{k_i}^{d_i}$ which are of codimension $w$, where $\tilde{k_i}$ is the residue field of $k_i$. Let $q_i=\#(\tilde{k_i})$. Then this number (which we denote by $n(q_i,d_i,w)$) is a gaussian $q_i$-binomial number:
 \begin{eqnarray*} n(q_i,d_i,w) = \binom{d_i}{w}_{q_i} &=& \frac{ ({q_i}^{d_i}-1)\cdot... \cdot ({q_i}^{d_i}-{q_i}^{w-1})}{({q_i}^w-1)\cdot .. \cdot ({q_i}^w-{q_i}^{w-1})} = \\
 &=& \frac{({q_i}^{d_i}-1)({q_i}^{{d_i}-1}-1)..({q_i}^{d_i-w+1}-1)}{({q_i}^w-1)({q_i}^{w-1}-1)..({q_i}-1)}.\end{eqnarray*}

Since $n({q_i},d_i,w)=n({q_i},d_i,d_i-w)$ and since these numbers (like the binomial numbers) increase with $w$ for $w=1,2,..,\lfloor(d_i+1)/2 \rfloor$, the set $\{w,d_i-w\}$ is uniquely determined by the number $n({q_i},d_i,w)$. 

The same enumeration holds for the set $B$. We conclude that a vertex of $A$ for which $t=a_i+w$ must be mapped to a vertex in $B$ for which $s=b_j\pm w$. In particular, setting $w=1$, we get that $p_j(t+1)=p_j(t)\pm 1$ for all $t$, and the result follows immediately.
\end{proof}

Let us call an automorphism of $\B_i$ \emph{label reversing} if it induces a permutation of the form $t \mapsto a-t$ on the labels of $\B_i$, and \emph{label preserving} otherwise. By the multiplicativity of the determinant, any element of $G_i$ acts on $\B_i$ by a label preserving automorphism. Let us now construct a label reversing automorphism of each $\B_i$. Let us fix on each $k_i$-vector space $V_i$ a nondegenerate bilinear form: $\langle \sum_j a_{i,j} T_{i,j}, \sum_j b_{i,j} T_{i,j} \rangle = \sum_j a_{i,j} b_{i,j}$. For a lattice $L$ of $V_i$, let $L^*$ be the dual of $L$: $$ L^* = \{ v\in V_i: \langle v,w\rangle \in O_{k_i} \text{ } \forall w\in L\} $$
Let us define an automorphism $\lambda_i$ of $\B_i$ by $\lambda_i([L])=[L^*]$. 
\begin{lemma} \label{involution}
\begin{enumerate}
\item $\lambda_i$ is a simplicial automorphism of $\B_i$.
\item $\lambda_i \circ \lambda_i = id$.
\item For $v=[\langle \pi_i^{n_0}T_{i,0}, .. ,\pi_i^{n_{d_i}}T_{i,d_i}\rangle] \in \Lambda_i$, we have $\lambda_i(v)=[\langle\pi_i^{-n_0}T_{i,0}, .. ,\pi_i^{-n_{d_i}}T_{i,d_i}\rangle]$. In other words, $\lambda_i$ acts on $\Lambda_i$ as the symmetry with respect to the origin $\langle T_{i,0},..,T_{i,d_i}\rangle$. 
\item $\lambda_i$ is label reversing. 
\item For a norm $\rho$ on $V_i$, we have $\lambda_i(\rho)=\rho^*$ where $ \rho^*(v)=\max_{0 \neq w\in V }\frac{|\langle v,w\rangle|}{\rho(w)}$.

\end{enumerate} \end{lemma}
\begin{proof}
If $\{[L_0],..,[L_k]\}$ is a $k$-dimensional face of the building, then we can rearrange the $L_i$'s and choose the representatives in the homothety classes in such a way that $L_0 \supset L_1 \supset .. \supset L_k
\supset \pi_i L_0$. This implies the inclusions $L_0^* \subset L_1^* \subset .. \subset L_k^* \subset {\pi_i}^{-1}L_0^*$. Hence, $L_k^* \supset L_{k-1}^* \supset .. \supset L_0^* \supset \pi_i L_k^*$ and we get that 
$\phi( \{[L_0],..,[L_k]\} )$ is a $k$-dimensional face as well. This shows the first statement.
The second statement follows from the standard fact that $L^{**}=L$. The third statement is trivial and the fourth one follows from the third. 
As for the last statement, let us first prove it when $[\rho]$ is a vertex of the building. In that case, there exists a lattice $L$ such that 
$\rho(v)=\min\{|c|: 0 \neq c\in k, v/c\in L\}$ for all $v \in V_i \setminus \{0\}$. Since $L^*=\{v: \langle v,w\rangle\in O_k \text{ } \forall w\in L\}$, we have: \begin{eqnarray*}
\rho_{L^*}(v)&=&\min\{|c|: 0 \neq c\in k_i, v/c\in L^*\}= \\
&=& \min\{|c|:0 \neq c\in k_i, |\langle v/c,w\rangle|\leq 1\text{ } \forall w\in L\} \end{eqnarray*}
hence
$$\rho_{L^*}(v)=\max_{w\in L}{|\langle v,w\rangle|}= \max_{w\in L \setminus \pi_i L}{|\langle v,w\rangle|} = \max_{0 \neq w\in V_i}\frac{|\langle v,w\rangle|}{\rho_L(w)}.$$
In general, if we have a face given by $L_0\supset L_1 \supset .. \supset L_k \supset L_{k+1}=\pi_i L_0$ and a norm $\rho$ which belongs to this face, satisfying $\rho(v)=R_i$ when $v\in L_i \setminus L_{i+1}$, for some numbers $R_0 > R_1 > .. > R_k > |\pi_i|R_0$, we will map $[\rho]$ to the point given by the face $L_k^* \supset L_{k-1}^* \supset .. \supset L_0^* \supset \pi L_k^*$  and the numbers $1/R_k > 1/R_{k-1} > .. > 1/R_0 > |\pi_i| / R_k$. This map extends linearly the map which we described on the vertices. It can be checked that the above point is equal to $[\rho^*]$.
\end{proof}

\begin{definition}
\label{sigma_mu}
\begin{enumerate}
\item
Let $1\leq i,j \leq r$ be such that $d_i=d_j$. Let $\phi^i_j: \Lambda_i \rightarrow \Lambda_j$ be the simplicial isomorphism mapping $[\langle \pi_i^{n_0} T_{i,0},..,\pi_i^{n_{d_i}} T_{i,d_i} \rangle]$ to $[\langle \pi_j^{n_0} T_{j,0},..,\pi_j^{n_{d_i}} T_{j,d_i} \rangle]$.
\item
For any permutation $\mu$ of $\{1,..,r\}$ such that $d_i = d_{\mu(i)}$ for all $i$, let $\sigma_\mu$ be the polysimplicial automorphism of $\Lambda$ defined by $\sigma_\mu(x_1,..,x_r) = (\phi^{\mu(i)}_i(x_{\mu(i)}))_i$.
\end{enumerate}
\end{definition}

\begin{corollary} \label{normalform}
Let $\phi \in \Aut_K(\X)$. Then there exist $g\in G$, numbers $r_i \in \{0,1\}$ for $1\leq i \leq r$, and a permutation $\mu$ of $\{1,..,r\}$ such that $d_{\mu(i)}=d_i$ for all $i$, and the automorphism $\phi'$ of $\B$ defined by $\phi':=(\prod_i \lambda_i ^ {r_i}) g \phi$ satisfies:
\begin{enumerate}
\item
$\phi'$ is polysimplicial, chambered and length preserving.
\item
$\phi'(\Lambda)=\Lambda$.
\item
$\phi'|_\Lambda = \sigma_\mu$.
\end{enumerate}
\end{corollary}

\begin{proof}
By proposition \ref{induce}, $\phi$ induces a polysimplicial, length preserving automorphism of $\B$. By lemma \ref{citebrt}, there exists $g\in G$ such that $g \phi (\Lambda)=\Lambda$ and $g \phi (\Delta)=\Delta$. 
Let $x_0$ be the origin of $\Lambda$, i.e the point $([L_1],..,[L_d])$ where $L_i = \langle T_{i,0},..,T_{i,d_i} \rangle$.
The linear endomorphism of $V_i$ defined by $f_i(T_{i,j})= T_{i,j-1}$ for $j>0$ and $f_i(T_{i,0})=\pi_i T_{i,d_i}$  induces an automorphism of $\B_i$ which preserves $\Lambda_i$ and $\Delta_i$, and performs a cyclic permutation on the vertices of $\Delta_i$. Multiplying $g$ by suitable powers of the $f_i$'s, we may assume that $g \phi$ fixes $x_0$. 
Let $\mu$ and the permutations $p_i$ be as in lemma \ref{lem:color} (applied on the automorphism $g \phi$), and define $r_i=0$ if $p_i$ is of the form $t \mapsto t+a_i$, and $r_i=1$ otherwise. By our assumption that $g\phi x_0=x_0$, we have $a_i=0$ for all $i$. Let $\phi'=(\prod_i \lambda_i ^ {r_i}) g \phi$. Then $C\circ \phi' \circ D$ is the permutation $\mu$ on $\prod_i \Z / d_i \Z$. It follows that $\phi'=\sigma_\mu$ on $\Delta$. By lemma \ref{aptaut}, $\phi'|_\Lambda = \sigma_\mu$.
\end{proof}

It would be useful to compose $\phi'$ with $\mu^{-1}$ and get an automorphism of $\B$ which fixes $\Lambda$ pointwise. However, this is not possible in a strict sense since the buildings $\B_i$ may not be interchanged. To allow such permutations, we need the ability to embed buildings defined over different fields in a common building. We describe such a construction in the next section.

% ************************* The building under a field extension ************************

\section{The building under a field extension}

Let $h/h'$ be an extension of non-Archimedean local fields, and $d\geq 1$ an integer. Let us discuss the relation between the buildings $B^d_h$ and $B^d_{h'}$ of $SL_{d+1}/\h$ and $SL_{d+1}/{h'}$ respectively. We fix an $h'$-vector space $W$ with a basis $w_0,..,w_d$ and view $B^d_{h'}$ as the set of homothety classes of norms on $W$, and $B^d_h$ as the set of homothety classes of norms on $W \otimes_{\h'} \h = \bigoplus_i h w_i$.

We have a map between the sets of vertices $$ \nu^d_{\h,\h'} : \text{ver}(B^d_{\h'}) \rightarrow \text{ver}(B^d_\h)$$ defined by
$ \nu^d_{\h,\h'}([L]) = [L\otimes_{O_{\h'}}O_\h] $ for an $O_{\h'}$-lattice $L$. Then $\nu^d_{h,h'}$ maps the vertices of an apartment of $B^d_{h'}$ to vertices in the apartment of $B^d_h$, which corresponds to the same basis. We extend $ \nu^d_{\h,\h'}$ to $B^d_{\h'}$, as an affine map on each chamber of $B^d_{\h'}$.

We also have a map in the other direction $\delta^d_{\h,\h'}: B^d_{\h} \rightarrow B^d_{\h'}$ defined by restricting a norm on $\h w_0 \oplus .. \oplus \h w_d$ to $\h'w_0 \oplus .. \oplus \h'w_d$. We have $\delta^d_{\h,\h'} \circ  \nu^d_{\h,\h'} = id$ (this is easily checked on the lattice $\langle w_1,..,w_d\rangle$, and by $PGL_{d+1}(\h')$-equivariance it holds for all the vertices of $B^d_{\h'}$). It follows that $\nu^d_{\h,\h'}$ is injective.

%We get the following two commutative diagrams:
%\[ \begin{array}{cccccccc}
%\h_{k'}^d&\supset                              &\h_k^d&&&\h_{k'}^d&\supset                              &\h_k^d\\
%\downarrow   &                                     &\uparrow  &&&\cup         &                                     &\cup  \\
%B^d(k')    &\overset{\delta_{d,l,l'}}{\twoheadleftarrow}&B^d(k)    &&&B^d_(k')    &\overset{\nu^d_{l,l'}}{\hookrightarrow}       &B^d(k)    \\
%\end{array}\]
Let us look at two extreme cases:

\begin{itemize}
\item The extension $\h/\h'$ is {\bf unramified}.
In this case, $\h$ and $\h'$ have the same uniformizer $\xi$.  Since $L\supset M \supset \xi L$ implies $\nu^d_{\h,\h'}(L)\supset \nu^d_{\h,\h'}(M) \supset \xi \nu^d_{\h,\h'}(L)$, $\nu^d_{\h,\h'}$ is a simplicial map. However, since the residue field of $\h$ is strictly bigger than the residue field of $\h'$, $\nu^d_{\h,\h'}$ is not surjective.

\begin{figure*}[h]
	\begin{center}
		\includegraphics[width=0.15\textwidth]{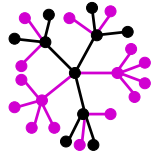}
	\end{center}
	\caption{The inclusion of buildings for an unramified quadratic extension ($d=1$,$p=2$)}
	\label{fig:twocolors}
\end{figure*}

\item The extension $\h/\h'$ is {\bf totally ramified}.
In this case, $\nu^d_{\h,\h'}$ is not a simplicial map. Two vertices at distance $1$ map to two vertices at distance $e(\h'/\h)$ (the ramification index). However, the buildings are (at least locally) isomorphic.

\begin{figure*}[h]
	\begin{center}
		\includegraphics[width=0.14\textwidth]{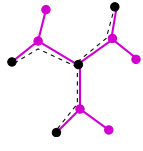}
	\end{center}
	\caption{The inclusion of buildings for a ramified quadratic extension ($d=1$,$p=2$)}
	\label{fig:ramified}
\end{figure*}

\end{itemize}

For a general extension $\h/\h'$, we get a superposition of the above two phenomena. In order to make the map $\nu^d_{\h,\h'}$ simplicial, we will refine the simplicial structure of $B^d_{\h'}$. We will consider two structures on $B^d_{\h'}$: One induced by its embedding in $B^d_{\h}$, and the other obtained by subdividing each edge into $e(\h/\h')$ parts, and subdividing all the faces of the building accordingly. Later on, we will show that these two structures coincide.

\begin{definition} Let $X$ be a polysimplicial complex, and $\F$ its set of faces. Let $Y$ be a set and $i:Y\rightarrow X$ an injective function. The induced polysimplicial structure on $Y$ is the set of faces $\{ F \subseteq Y \, | \, i(F) \in \F \}$, with the order relation and simplicial structure inherited from those of $X$.
\end{definition}

Note that this definition does not guarantee that the resulting structure is a polysimplicial complex. 

Next, we define the subdivision of a $d$-simplex, where each edge is divided into $N$ parts.

For an integer $d\geq 1$, let us look at an apartment $\Sigma$ in the building of $SL_{d+1}$ over any local field. It is isomorphic to $\R^{d+1} / \R \cdot (1,1,..,1)$, and the vertices correspond to the image of $\Z^{d+1}$ in this space. Hence, we have coordinates $x_0,..,x_d$ on $\Sigma$. Let $\eta$ be the chamber in $\Sigma$ given by the inequalities $x_0 \leq x_1 \leq .. \leq x_d \leq x_0 + 1$. Each face of $\eta$ is defined by a finite set of conditions of the form $x_i=x_{i+1}$ (for $0\leq i \leq d-1)$ or $x_d=x_0+1$. The other chambers of $\Sigma$ can be obtained from $\eta$ by permutations of the coordinates and translations:
Given a permutation $\sigma$ of $\{0,..,d\}$ and integers $a_0,..,a_d$, we have a chamber
\begin{equation*}
\begin{split}
 \eta(\sigma,&a_0,..,a_d) := \\
 &\{ [(x_0,..,x_d)]: x_{\sigma(0)}+a_0 \leq x_{\sigma(1)}+a_1 \leq .. \leq x_{\sigma(d)}+a_d \leq x_{\sigma(0)}+a_0+1\}
\end{split}
\end{equation*}

For any point $[x] \in \R^{d+1} / \R \cdot (1,1,..,1)$, where $x=(x_0,..,x_d)$, we may sort the fractional parts of the coordinates $x_0,..,x_d$ in nondecreasing order, and obtain in this way a chamber $\eta(\sigma,a_0,..,a_d)$  in which $[x]$ is contained. Hence, any chamber of $\Sigma$ if of this form.

For an integer $N$, let $\eta_N$ be the dilated simplex
$$ \eta_N=  \{[(x_0,..,x_d)]: x_0 \leq x_1 \leq .. \leq x_d \leq x_0 + N\} $$

\begin{prop} 
$\eta_N$ is a finite union of chambers of $\Sigma$. 
\end{prop}
\begin{proof} As $\eta_N$ is a closed set, and the union of the interiors of the chambers is dense in the apartment, it is enough to show that given a chamber $F$ and an interior point $[x]$ of $F$, if $[x]$ is in $\eta_N$ then $F \subseteq \eta_N$. Let us suppose that $F = \eta(\sigma, a_1,..,a_d)$. Then $x$ satisfies $$x_{\sigma(0)}+a_0 < x_{\sigma(1)}+a_1 < .. < x_{\sigma(d)}+a_d < x_{\sigma(0)}+a_0 + 1$$ and also $x_0 \leq x_1 \leq .. \leq x_d \leq x_0+N$. Let $[y] \in F$. We have  $$y_{\sigma(0)}+a_0 \leq y_{\sigma(1)}+a_1 \leq .. \leq y_{\sigma(d)}+a_d \leq  y_{\sigma(0)}+a_0 + 1$$ Let us show that $y_0 \leq .. \leq y_d \leq y_0+N$.

Let $0 \leq i \leq d-1$, and write $i=\sigma(m)$, $i+1=\sigma(n)$. We consider two cases:
\begin{itemize}
\item
If $m<n$ then $x_i+a_m < x_{i+1}+a_n < x_i + a_m + 1$. Since $x_i \leq x_{i+1}$, we get $x_{i+1}+a_n < x_{i+1} + a_m + 1$. Hence $a_n < a_m+1$, hence $a_n \leq a_m$.
Since $[y] \in F$, we have $$y_i +a_m \leq y_{i+1} + a_n \leq y_{i+1} + a_m \Rightarrow y_i \leq y_{i+1}$$
\item
If $m>n$ then $x_i + a_m > x_{i+1}+a_n \geq x_i + a_n$. Hence, $a_m > a_n$. We conclude that
$$ y_i + a_m \geq y_{i+1}+a_n \geq y_i+a_m-1 \geq y_i+a_n \Rightarrow y_{i+1}\geq y_i $$
\end{itemize}
as desired. A similar analysis shows that we always have $y_d \leq y_0+ N$.
\end{proof}

\begin{definition} Given a $d$-dimensional closed simplex $F$ and an integer $N$, we let $F[N]$ be the simplicial complex whose underlying set is $F$, and the simplicial structure is obtained by choosing an affine isomorphism $f: F \rightarrow \eta_N$, and taking the induced simplicial structure on $F$. \end{definition}

It is easy to see that this definition is independent of the choice of $f$. 

\begin{lemma}\label{compatible}
Let $F$ be a $d$-dimensional simplex and let $F'$ be a closed face of $F$. Then the simplicial structure of $F'[N]$ coincides with the simplicial structure induced from $F[N]$.
\end{lemma}
\begin{proof}
It is enough to verify the claim when $F'$ is of codimension 1 in $F$. Also, it is enough to prove the claim for one specific simplex and one specific facet of it. Let us take $F = \eta_N$ and $F'$ the facet defined by the hyperplane $x_0=x_1$. Let $\kappa'$ be a chamber in $F'[N]$. Then $\kappa'$ has the form
\begin{equation*}
\begin{split}
\kappa' = \{[(x_0,..,x_d)]:& x_0=x_1, \\
                                     &x_{\sigma(1)} + a_1 \leq x_{\sigma(2)} + a_2 \leq .. \leq x_{\sigma(d)} + a_d \leq x_{\sigma(1)} + a_1 + 1 \} 
\end{split}
\end{equation*}
for some permutation $\sigma$ of $\{1,..,d\}$, and integers $a_1,..,a_d$. 
Let us prove that such a chamber is a closed face of a chamber of $F[N]$: Let $m$ be such that $\sigma(m)=1$, and let $\kappa_1$ and $\kappa_2$ be the chambers of the apartment $\Sigma$ defined the same way as $\kappa'$, but with the equality $x_0=x_1$ replaced by $x_{\sigma(m-1)}+a_{m-1} \leq x_0+a_m \leq x_1 + a_m$ and $x_1+a_m \leq x_0+a_m \leq x_\sigma(m+1)+a_{m+1}$, respectively. Then $\kappa'$ is a common closed face of $\kappa_1$ and $\kappa_2$, of codimension 1. Thus, $\kappa_1$ and $\kappa_2$ are neighbouring chambers, and each of them is the reflection of the other with respect to the hyperplane $x_1=x_2$. In this situation, it is easy to prove that any interior point of $\kappa'$ is an interior point of $\kappa_1 \cup \kappa_2$. Let us take such an interior point $x$, and a neighbourhood $A$ of $x$ which is contained in $\kappa_1 \cup \kappa_2$. Since $x\in F$, There is a nonempty intersection between $A$ and the interior of $F$. Let $y$ be a point in $A \cap F$. Then $y$ is either in $\kappa_1$ or in $\kappa_2$. According to the proof of the last proposition, either $\kappa_1$ or $\kappa_2$ are contained in $F$. Let us suppose that $\kappa_i \subseteq F$. Then $\kappa_i$ is a chamber of $F[N]$, such that $\kappa_i \cap F' = \kappa'$, as desired.
 \end{proof}

\begin{definition} Let $B$ be a simplicial complex. We let $B[N]$ be the simplicial complex with the same underlying set as $B$, where each closed simplex $F$ of $B$ is given the simplicial structure $F[N]$.
 \end{definition}

Lemma \ref{compatible} guarantees that the symplicial structures $F[N]$ glue to a simplicial complex structure.

\begin{lemma} \label{twostructs}
Let $\h/\h'$ be an extension of complete local fields, and let $e$ be its ramification index. Then the simplicial structure induced from $B^d_\h$ on $B^d_{\h'}$, via the embedding $\nu^d_{\h,\h'}$, coincides with $B^d_{\h'}[e]$. 
\end{lemma}

\begin{proof} Let $\xi$ and $\xi'$ be uniformizers of $\h$ and $\h'$, respectively, such that $\xi' = \xi ^ e$. Let $A'$ be the apartment of $B^d_{\h'}$ corresponding to the basis $w_0,..,w_d$, and let $A$ be the apartment of $B^d_{h}$ corresponding to the same basis. The map $\nu^d_{\h,\h'}$ induces an isomorphism from $A'$ to $A$. 

Let $\Theta$ be the basic chamber of $B^d_{\h'}$, 
$$\Theta= \{\rho\in A' \, |\,  \rho(w_0) \geq \rho(w_1) \geq ... \geq \rho(w_d) \geq |\xi'| \rho(w_0) \}$$

Since the embedding $\nu^d_{\h,\h'}$ is $PGL_{d+1}(h')$-equivariant, it is enough to verify that the induced structure on $\Theta$ is $\Theta[e]$. 

Indeed, let us map $A$ to the apartment $\Sigma$ defined above by $\langle \xi^{m_0} w_0,..,\xi^{m_d} w_d \rangle \mapsto [-(m_0, m_1,..,m_d)]$. This map is a simplicial isomorphism, and under this map  $\nu^d_{\h,\h'}(\Theta)$ is mapped to the dilated simplex $\eta_e$ defined above. Therefore the induced structure on $\Theta$ is $\Theta[e]$.\end{proof}

Next, we will describe a way of obtaining a subdivision of a polysimplicial complex, given numerical data called \emph{marking}. 
\begin{definition}\begin{enumerate} \item Let $A$ be a polysimplicial complex. A marking on $M$ on $A$ is an assignment of a natural number $M(e)$ to each edge $e$ of $A$, such that for every closed face $F$ of $A$, and a representation of $F$ as a product of closed simplices $F=\prod_{i=1}^nF_i$, there exist natural numbers $a_1,..,a_n$ such that each edge of $F$ that comes from $F_i$ is assigned the number $a_i$ by $M$.

\item Given $A$ and $M$, we define a polysimplicial complex $A[M]$ as follows: The underlying set of $A[M]$ is the same as of $A$. For each closed face $F$ of $A$, we write it as a product of closed simplices: $F = \prod F_i$. Let $a_1,..,a_n$ be the numbers guaranteed by the definition of $M$. Then we take the polysimplicial structure $\prod_{i=1}^n F_i[a_i]$ on $F$.
\end{enumerate}
\end{definition}

Using the first part of lemma \ref{graphaut} and lemma \ref{compatible}, it is easy to see that the polysimplicial structures defined on the closed faces of $A$ in the above definition are compatible, hence glue to a polysimplicial structure on $A$. 

\begin{prop}\label{extend} Let $A_1, A_2$ be polysimplicial complexes. Let $M_1, M_2$ be markings on $A_1$ and $A_2$ respectively, and let $f$ be a polysimplicial chambered map from $A_1$ to $A_2$ such that $f$ preserves the marking: $M_2(f(e))=M_1(e)$ for any edge $e$ of $M_1$. Then, $f$ is a polysimplicial chambered map from $A_1[M_1]$ to $A_2[M_2]$, as well. \end{prop} 
\begin{proof} If $A_1$ and $A_2$ are closed simplices, then $f$ is merely a permutation of the vertices extended linearly, hence $f$ extends to an isomorphism from $A_1[N]$ to $A_2[N]$ for any $N$. If $A_1$ and $A_2$ are closed polysimplices, then by lemma \ref{graphaut} there exist representations $A_1=\prod F_i$, $A_2 = \prod F'_i$, where $F_i, F'_i$ are simplices of the same dimension for each $i$, and $f$ induces a bijection from $F_i$ to $F'_i$. Since $f$ preserves the marking, there exist numbers $N_i$ such that $A_1[M_1]=\prod F_i[N_i]$ and $A_2[M_2]=\prod F'_i[N_i]$. Hence, $f$ extends to an isomorphism from $A_1[M_1]$ to $A_2[M_2]$, by the previous case. In the general case, since $f$ is a polysimplicial isomorphism on $F[N]$ for each chamber $F$ of $A_1$, it is polysimplicial and chambered on the whole of $A_1$. \end{proof}

We now come back to the building $\B=\prod_i{\B_i}$, where $\B_i=B^{d_i}_{k_i}$. Since all the fields $k_i$ are contained in $k$, we have embeddings $\nu_i=\nu^{d_i}_{k_i,k}:\B_i \rightarrow \B_i(k)$, where $\B_i(k)=B^{d_i}_k$. We combine these embedding to an embedding 
$\nu:\B\rightarrow \B(k)=\prod_i \B_i(k)$. 

Let $\tilde{\B}$ be the building $\B$ with the polysimplicial structure induced from the embedding $\nu: \B \rightarrow \B(k)$.

\begin{prop}\label{subdivide}
\begin{enumerate}
\item $\tilde{\B}$ is a polysimplicial complex. 
\item $\nu:\tilde{\B} \rightarrow \B(k)$ is a polysimplicial embedding.
\item If $\phi$ is a length preserving, polysimplicial, chambered automorphism of $\B$, then it is a polysimplicial chambered automorphism of $\tilde{\B}$.
\end{enumerate}
\end{prop}

\begin{proof} Let $M$ is the marking on $\B$ which assigns to an which comes from $\B_i$ the number $e(k/k_i)=\log(|\pi_i|)/\log |\pi|$. By lemma \ref{twostructs}, $\tilde{\B}=\B[M]$. Hence, it is a polysimplicial complex. The second assertion is immediate, by the definition of $\tilde{\B}$. Since the automorphism $\phi$ is length preserving, it preserves the marking $M$. By proposition \ref{extend}, $\phi$ is a polysimplicial, chambered automorphism of $\B[M]=\tilde{\B}$. \end{proof}

% *********************** The commutation with the projection to an apartment ***************

\section{The commutation with the projection to an apartment}

For any $1\leq i \leq r$, let us define (following \cite{ber}) a projection map $\tau_{\Lambda_i}:\B_i \rightarrow \Lambda_i$: Given a norm $\rho$ on $V_i$, we define $\tau_{\Lambda_i}([\rho])=[\tilde{\rho}]$, where $\tilde{\rho}(\sum_j a_j T_{i,j})= \max_j |a_j|\rho(T_{i,j})$. By \cite{ber}, 5.1, $\tau_{\Lambda_i}$ is a simplicial map.

Combining these projections, we get a projection map $\tau_{\Lambda}:\B \rightarrow \Lambda$. 

Let $\Lambda(k)=\nu(\Lambda)$. $\Lambda(k)$ is an apartment of $\B(k)$, equal to $\prod_i \Lambda_i(k)$, where $\Lambda_i(k)$ is the apartment in $\B_i(k)$ corresponding to the basis $T_{i,0},..,T_{i,d_i}$. We have a projection $\tau_{\Lambda(k)}:\B(k) \rightarrow \Lambda(k)$, defined in the same way as $\tau_\Lambda$.

We also group the restriction maps $\delta^{d_i}_{k,k_i}:\B_i(k) \rightarrow \B_i$, defined in the previous section, to a single map $\delta: \B(k) \rightarrow \B$. We have $\delta \circ \nu = id$.

The image of the equivalence class of a norm $\rho$ under $\tau_{\Lambda}$ is determined by the values $\rho(T_{i,j})$. Hence, it follows from the definition that we have a commutative diagram 

\[\begin{array}{ccc}
\B &\overset{\delta}{\leftarrow}&\B(k) \\
\tau_\Lambda \downarrow & & \downarrow \tau_{\Lambda(k)}\\
\Lambda & \overset{\delta}{\xleftarrow{\sim}} & \Lambda(k) \\
\end{array}\]

From which follows the commutative diagram

%\[\begin{array}{ccccccc}
%\B &\overset{\nu}{\hookrightarrow}&\B(k) & & \B &\overset{\nu}{\hookrightarrow}&\B(k)\\
%\tau_\Lambda \downarrow & & \downarrow \tau_\Lambda & &  \uparrow & & \uparrow \\
%\Lambda & = & \Lambda & & \Lambda & = & \Lambda\\
%\end{array}\]

\[\begin{array}{ccc}
\B &\overset{\nu}{\hookrightarrow}&\B(k)\\
\tau_\Lambda \downarrow & & \downarrow \tau_{\Lambda(k)}\\
\Lambda &  \overset{\nu}{\xrightarrow{\sim}} & \Lambda(k)\\
\end{array}\]

In these diagrams, we may replace $\B$ by $\tilde{\B}$, as the two complexes have the same underlying sets.

\begin{prop} \label{commute} Let $\psi: \tilde{\B} \rightarrow \B(k)$ be a polysimplicial, chambered, injective map, such that the following diagram is commutative:
\[\begin{array}{ccc}
\tilde{\B} &\overset{\psi}{\rightarrow}&\B(k) \\
\uparrow & & \uparrow \\
\Lambda & \overset{\nu}{\rightarrow} & \Lambda(k) \\
\end{array}\]
Then the diagram
\[\begin{array}{ccc}
\tilde{\B} &\overset{\psi}{\rightarrow}&\B(k) \\
\tau_\Lambda \downarrow & & \downarrow \tau_{\Lambda(k)}\\
\Lambda &  \overset{\nu}{\rightarrow} & \Lambda(k) \\
\end{array}\]
is commutative too.
\end{prop}

\begin{proof}
Let us first understand better the projection $\tau_{\Lambda(k)}$. 
Let us define a function $f_i$ on pairs of vertices of $\B_i(k)$ by $$f_i([M],[L])=[M:L]$$ where the lattices $M$ and $L$ are chosen so that $M \supseteq L$ and $\pi_i M \nsupseteq L$, and $[M:L]$ means the length of $M/L$ as an $O_{k_i}$-module. Note that this is not a symmetric function. It is proven in \cite{ber} that for a vertex $x$ of $\mathcal{B}_i(k)$, $\tau_{\Lambda_i(k)}(x)$ is the unique vertex $y$ of $\Lambda_i(k)$ for which $f_i(x,y)$ is minimal.
Let us give an equivalent definition of $f_i$. We endow the set of vertices of $\B_i(k)$ with the structure of a directed graph in the following way: There is an edge from $x$ to $y$ if and only if $f_i(x,y)=1$. 
\begin{lemma}\label{minlen}
For any vertices $x,y$ of $\B_i(k)$,  $f_i(x,y)$ is the length of the minimal directed path from $x$ to $y$. 
\end{lemma}
\begin{proof}
Let $M$ and $L$ be lattices of $V_i$ such that $x=[M]$, $y=[L]$,  $M \supseteq L$ and $\pi_i M \nsupseteq L$. Then $f(x,y)=[M:L]$ and we can replace the inclusion $M\supseteq L$ by a chain of $[M:L]$ inclusions of relative degree $1$. This shows that the distance between $x$ and $y$ is at most $f(x,y)$. On the other hand, any directed path from $x$ and $y$ can be represented by lattices $M=M_0 \supset M_1 \supset .. \supset M_k$ such that $[M_j:M_{j+1}]=1$ for all $j$, and $M_k= \pi_i^n L$ for some $n>0$. Hence $k=[M:M_k] \geq [M:L]$. We conclude that the minimal path length is $[M:L]$.
\end{proof}
The vertices of $\B_i(k)$ also have an undirected graph structure, coming from the 1-skeleton of $\B_i(k)$: Two vertices $x$ and $y$ are connected in this graph if they can be represented by lattices $L$ and $M$ such that $L \supseteq M \supseteq \pi_i L$. 

Let us endow $\B_i(k)$ with a labelling $C_i$ with values in $\Z / d_i \Z$ as in definition \ref{def:label}. Then, for $x,y\in \B_i(k)$, there is an edge from $x$ to $y$ in the directed graph if and only if there is an edge between $x$ and $y$ in the undirected graph, and $C_i(y)=C_i(x)+1$. 

Similarly, for two vertices $x=([M_1],..,[M_r])$ and $y=([L_1],..,[L_r])$ of $\B(k)$, let us define $f(x,y) = \sum_{i=1}^r{f_i([M_i],[L_i])}$. Then clearly $\tau_{\Lambda(k)}(x)$ is the vertex $y\in \Lambda(k)$ for which $f(x,y)$ is minimal. We  also endow the vertices of $\B(k)$ with the product directed graph structure (there is an edge from $x$ to $y$ when there exists $i$ such that there is an edge between the $i^\text{th}$ coordinate of $x$ and the $i^\text{th}$ coordinate of $y$ in the directed graph of $\B_i(k)$, and the other coordinates are identical in $x$ and $y$), and we let $\B(k)$ have the product labelling, with values in $\prod_i \Z / d_i \Z$. 
There is an edge from a vertex $x$ of $\B(k)$ to a vertex $y$ if and only if $x$ and $y$ are connected in the $1$-skeleton of $\B(k)$ and $C(y)=C(x)+(0,0,..0,\underset{i}{1},0,..0)$ for some $i$. By lemma \ref{minlen}, $f(x,y)$ is equal to the length of a minimal directed path from $x$ to $y$. 
\begin{lemma}
For two vertices of $\B(k)$, $x$ and $y$, $f(x,y)$ is equal to the length of a minimal directed path from $x$ to $y$ in any apartment of $\B(k)$ which contains $x$ and $y$.
\end{lemma}\begin{proof}
It is enough to prove that for two vertices of $\B_i(k)$, $x$ and $y$, $f_i(x,y)$ is equal to the length of a minimal path from $x$ to $y$ in any apartment of $\B_i(k)$ which contains $x$ and $y$.
Let $v_0,..,v_{d}$ be a basis of such an apartment (where $d=d_i$), and let $L$ and $M$ be two lattices such that$x=[M]$, $y=[L]$, $M \supseteq L$ and $\pi_i M \nsupseteq L$. Then we can write $M=\langle{\pi_i}^{m_0}v_0,..,{\pi_i}^{m_{d_i}}v_{d_i}\rangle, L=\langle{\pi_i}^{l_0}v_0,..,{\pi_i}^{l_{d_i}}v_{d_i}\rangle$. We have $m_j \geq l_j$ for all $j$, $[M:L]=\sum_j (m_j-l_j)$. It is easy to find a path of length $[M:L]$ from $[M]$ to $[L]$ in the apartment (decreasing each time the exponent of $\pi_i$ in one of the coordinates by 1. By lemma \ref{minlen}, the proof is complete.
\end{proof}

We let $\tilde{\B}$ inherit from $\B(k)$, in addition to its polysimplicial structure, the labelling and the directed graph structure. Then it is clear that the directed graph structure is determined by the 1-skeleton and the labelling in the same way. For two vertices $x$ and $y$ of $\tilde{\B}$, if $A$ is an apartment of $\tilde{\B}$ containing both vertices, then by \cite{brt}, corollary 2.8.6, $\nu(A)$ is an apartment of $\B(k)$ containing $\nu(x)$ and $\nu(y)$. It follows from the above lemma that $f(\nu(x), \nu(y))$ is equal to the length of the minimal path from $x$ to $y$ in $\tilde{\B}$. We call this function $f(x,y)$ as well. Hence $\tau_\lambda(x)$ is the vertex $y$ of $\Lambda$ which minimizes $f(x,y)$.

By the assumption on $\psi$, for any vertex $x\in \Lambda$ we have $\psi(x)=\nu(x)$, hence $C(\psi(x))=C(\nu(x))=C(x)$. The two labellings  $C$ and $C \circ \psi$ of $\tilde{\B}$ are therefore identical on $\Lambda$. By lemma \ref{lem:label}, they are identical on $\tilde{\B}$. We conclude that $\psi$ preserves the labelling $C$. Since it clearly preserves the 1-skeleton structure, we get that $\psi$ preserves the directed graph structure.  

Let $x$ and $y$ be vertices of $\tilde{\B}$, and let $A$ be an apartment of $\tilde{\B}$ which contains $x$ and $y$. By \cite{brt}, loc. cit., $\psi(A)$ is an apartment of $\B(k)$ which contains $\psi(x)$ and $\psi(y)$. Since $\psi$ preserves the directed graph structure, $\psi$ carries a path from $x$ to $y$ in $A$ to a path from $\psi(x)$ to $\psi(y)$ in $\psi(A)$. By the above lemma, we conclude that $f(\psi(x),\psi(y))=f(x,y)$.

Now, for a vertex $x_0$ of $\tilde{\B}$, let $t=\tau_\Lambda(x)$. Then $t$ minimizes the function $g(y):=f(x_0,y)$ on the vertices of $\Lambda$. Since $\psi=\nu$ on the vertices of $\Lambda$, we have $g(y)=f(\psi(x_0),\psi(y))=f(\psi(x_0),\nu(y))$ for all $y \in \Lambda$. Hence, $\nu(t)$ minimizes the function $h(y):=g(\psi(x_0),y)$ on $\nu(\Lambda)=\Lambda(k)$. This proves that $\tau_{\Lambda(k)}(\psi(x))=\nu(\tau_\Lambda(x))$ and finishes the proof of proposition \ref{commute}. \end{proof}

% *************** The end of the proof ******************************

\section{Proof of the main theorem}

We are now ready to prove theorem \ref{T}. Let $\phi \in \Aut_K(\X)$. By proposition \ref{induce}, $\phi(\B)=\B$, and $\phi$ commutes with $\tau:\X \rightarrow \B$. By corollary \ref{normalform}, there exist $g\in G$, $r_i \in \{0,1\}$ for $1\leq i \leq r$, and a permutation $\mu$ of $\{1,..,r\}$ such that $d_{\mu(i)}=d_i$ for all $i$, and the automorphism $\phi'$ of $\B$ defined by $\phi':=(\prod_i \lambda_i ^ {r_i}) g \phi$ satisfies:
\begin{enumerate}
\item
$\phi'$ is polysimplicial, chambered and length preserving.
\item
$\phi'(\Lambda)=\Lambda$.
\item
$\phi'|_\Lambda = \sigma_\mu$ (see definition \ref{sigma_mu}).
\end{enumerate}

It is enough to prove that $g\phi$ is of the desired form, so we assume that $g=id$, without loss of generality. By proposition \ref{subdivide}, $\phi'$ is a polysimplicial and chambered automorphism of $\tilde{\B}$. 

Since $\B(k)=\prod_i B^{d_i}_k$, $\mu$ acts on $\B(k)$ by a permutation of the factors. We have the following commutative diagrams of polysimplicial maps:

\[\begin{array}{ccc}
\Lambda &\overset{\sigma_\mu}{\rightarrow}&\Lambda\\
\nu \downarrow & & \downarrow \nu \\
\Lambda(k) &  \overset{\mu}{\rightarrow} & \Lambda(k)\\
\end{array}\]

and

\[\begin{array}{ccccccc}
  \tilde{\B} &\overset{\phi'}{\rightarrow}&\tilde{\B}&\overset{\nu}{\hookrightarrow}&\B(k) &\overset{\mu^{-1}}{\rightarrow}&\B(k)\\
  \uparrow & & \uparrow & & \uparrow & & \uparrow \\
  \Lambda &\overset{\sigma_{\mu}}{\rightarrow}&\Lambda& \overset{\nu}{\rightarrow} &\Lambda(k) &\overset{\mu^{-1}}{\rightarrow}&\Lambda(k)\\
 \end{array}\]

Where the vertical arrows are inclusion maps. 

Since $\nu \sigma_\mu = \mu \nu$, the composition of the lower row in the above diagram is equal to $\nu$. 

Let $\psi$ be the composition of the upper row: $$\psi = \mu^{-1} \circ \nu \circ \phi' = \mu^{-1} \circ \nu \circ (\prod_i \lambda_i ^ {r_i}) \circ \phi$$ 
We get the commutative diagram
\[\begin{array}{ccc}
\tilde{\B} &\overset{\psi}{\rightarrow}&\B(k) \\
\uparrow & & \uparrow \\
\Lambda & \overset{\nu}{\rightarrow} & \Lambda(k) \\
\end{array}\]
By lemma \ref{commute}, the vertical arrows can be reversed, to form the commutative diagram

\[\begin{array}{ccc}
\tilde{\B} &\overset{\psi}{\rightarrow}&\B(k) \\
\tau_\Lambda \downarrow & & \downarrow \tau_{\Lambda(k)}\\
\Lambda & \overset{\nu}{\rightarrow} & \Lambda(k) \\
\end{array}\]

Hence, we have a commutative diagram (in which the composition of the middle row is $\psi$):
$$
\begindc{\commdiag}[1]
\obj(20,140)[A]{$\X$}
\obj(60,140)[B]{$\X$}
\obj(20,100)[D]{$\tilde{\B}$}
\obj(60,100)[E]{$\tilde{\B}$}
\obj(100,100)[F]{$\tilde{\B}$}
\obj(140,100)[G]{$\B(k)$}
\obj(180,100)[H]{$\B(k)$}
\obj(20,60)[C]{$\Lambda$}
\obj(100,60)[I]{$\Lambda$}
\obj(140,60)[J]{$\Lambda(k)$}
\obj(180,60)[K]{$\Lambda(k)$}
\mor{A}{B}{$\phi$}
\mor{D}{E}{$\phi$}
\mor{E}{F}{$\prod\limits_i \lambda_i ^ {r_i}$}
\mor{F}{G}{$\nu$}[\atright, \injectionarrow]
\mor{G}{H}{$\mu^{-1}$}
\mor{D}{C}{$\tau_{\Lambda}$}
\mor{F}{I}{$\tau_{\Lambda}$}
\mor{H}{K}{$\tau_{\Lambda(k)}$}
\mor{G}{J}{$\tau_{\Lambda(k)}$}
\mor{I}{J}{$\nu$}
\mor{J}{K}{$\mu^{-1}$}
\mor{A}{D}{$\tau$}
\mor{B}{E}{$\tau$}[\atleft, \solidarrow]
\cmor((20,50)(30,40)(60,30)(100,30)(140,30)(170,40)(180,50)) \pup(100,20){$\nu$}
\enddc
$$

It follows that 

\begin{equation*}\mu^{-1}\circ \nu \circ \tau_\Lambda \circ \prod\limits_i \lambda_i ^ {r_i} \circ \tau \circ \phi = \nu \circ \tau_\Lambda \circ \tau \end{equation*}

Hence

\begin{equation}\label{finaleq}\mu^{-1}\circ \nu \circ \tau_\Lambda \circ \prod\limits_i \lambda_i ^ {r_i} \circ \tau = \nu \circ \tau_\Lambda \circ \tau \circ \phi^{-1}\end{equation}

Let us write this equality more explicitely. Note that all the objects in the commutative diagram are spaces of equivalence classes of norms or seminorms, and on each space the real valued functions $|t_{i,j}([\rho])|=\frac{\rho(T_{i,j})}{\rho(T_{i,0})}$ are well defined. These functions commute with the maps $\tau, \tau_{\Lambda}, \tau_{\Lambda(k)}$ and $\nu$. On the apartments $\Lambda$ and $\Lambda(k)$ these functions are coordinates that define the isomorphism to a Euclidean space. We will now evaluate the functions $|t_{i,j}|$ on both sides of (\ref{finaleq}).

Let $x\in \X$, $1\leq i \leq r$ and $1 \leq j \leq d_i$. Then
$$ |t_{i,j}((\nu \circ \tau_\Lambda \circ \tau \circ \phi^{-1})(x))| = |t_{i,j}(\phi^{-1}(x))|$$
On the other hand,
$$ |t_{i,j}((\nu \circ \tau_\Lambda \circ \prod\limits_i \lambda_i ^ {r_i} \circ \tau)(x))|=|t_{i,j}(\lambda_i^{r_i}(\tau_i(\pi_{X_i}(x))))|$$
If $r_i=0$, we have $$ |t_{i,j}((\nu \circ \tau_\Lambda \circ \prod\limits_i \lambda_i ^ {r_i} \circ \tau)(x))|=|t_{i,j}(x)|$$
Let us assume now that $r_i=1$. Let $x=[\rho]$, where $\rho$ is a multiplicative seminorm on $K[T_{i,j}]$. Then $\tau_i(\pi_{\X_i}(x))=[\rho|_{V_i}]\in \B_i$. By lemma \ref{involution}, 
$\lambda_i(\tau_i(\pi_{\X_i}(x)))=\lambda_i([\rho|_{V_i}])=[\rho']$, where $\rho'(v)=\max\limits_{0\neq w \in V} \frac {|\langle v,w\rangle|}{\rho(w)}$. Let us define 
\begin{eqnarray*} S_{i,j}(\rho) &=& \rho'(T_{i,j})=\max\limits_{0\neq w=\sum a_k T_{i,k} \in V_i} \frac {|\langle T_{i,j},w\rangle|}{\rho(w)}  =\\
&=& \max \frac{|a_j|}{\rho(\sum a_k T_{i,k})} = \max\limits_{a_1,..,\hat{a_j}, .., a_{d_i} \in k_i} \frac1{\rho(T_{i,j} + \sum\limits_{k \neq j}a_kT_{i,k})}\end{eqnarray*}
and $$s_{i,j}([\rho])=\frac{S_{i,j}(\rho)}{S_{i,0}(\rho)}$$
Then $s_{i,j}$ is a well defined real valued function on $\X$, and$$ |t_{i,j}((\nu \circ \tau_\Lambda \circ \prod\limits_i \lambda_i ^ {r_i} \circ \tau)(x))|=|t_{i,j}([\rho'])|=s_{i,j}(x)$$
This implies, for all $1\leq i \leq r$ and $1\leq j \leq d_i$, 
\[
|t_{i,j}(\phi^{-1}(x))| = \left \{ \begin{array}{cl}
|t_{\mu^{-1}(i),j}(x)| & r_i= 0 \\
s_{\mu^{-1}(i),j}(x) & r_i= 1 \end{array} \right.
\]
Let us introduce the action of the torus of $G$ on $\X$. For $1\leq i \leq r$, $0\leq j\leq d_i$ and $\alpha\in k_i^*$, let $m_{i,j,\alpha}$ be the automorphism of $\X$ which multiplies the $(i,j)$ coordinate $T_{i,j}$ by $\alpha$. Let $T$ be the torus generated by all these automorphisms, for $j>0$. Let us note that
\[
S_{i,j}(m_{i',j',\alpha}(x))= \left \{ \begin{array}{ll}
|\alpha|^{-1}S_{i,j}(x) & (i,j)=(i',j') \\
S_{i,j}(x) & (i,j) \neq (i',j') \end{array} \right.
\]
Hence, the same identities hold for the $s_{i,j}$'s, for $j,j' > 0$.
Obviously we have for the $t_{i,j}$'s
\[
|t_{i,j}(m_{i',j',\alpha}(x))|= \left \{ \begin{array}{ll}
|\alpha|\cdot|t_{i,j}(x)| & (i,j)=(i',j') \\
|t_{i,j}(x)| & (i,j) \neq (i',j') \end{array} \right.
\]
whenever $j,j'>0$. It follows that the functions
$$ u_{i,j}(x) := t_{i,j}(\phi^{-1}(x)) \cdot t_{\mu^{-1}(i),j}^{(-1)^{1+r_i}} $$
satisfy that $|u_{i,j}|$ is invariant under the action of $T$. 

\begin{lemma} Any bounded analytic function on $\X$ is constant. \end{lemma} \begin{proof} The case $r=1$ is proved in \cite{ber}, lemma 3. The general case follows easily by induction on $r$: Any such function must be constant when one of the coordinates is fixed to be a rigid point of one of the $\X_i$'s, by the inductive hypothesis. Hence the function is constant on the rigid points of $\X$, and therefore (by \cite{ber2}, proposition 2.1.15), on the entire space. \end{proof} 

For any $g\in T$, the function $\frac{u_{i,j}(gx)}{u_{i,j}(x)}$ is bounded, of absolute value $1$, hence it is constant. If we restrict $u_{i,j}(x)$ to the multi-annulus $\tau^{-1}(\Delta^{\circ})$ and develop it as a power series in the $t_{i,j}$s
$$ u_{i,j}(x) = \sum a_I t^I, $$ we have $$ u_{i,j}(gx) = \alpha_g u_{i,j}(x) $$ for a constant $\alpha_g\in K$ of absolute value $1$. This implies that for any $g=(g_{i,j})$, and any multi-index $I$, we have $g^I a_I = \alpha_g a_I$, hence if $a_I \neq 0$, then $\alpha_g = g^I$. Since $\alpha_g$ does not depend on $I$, $a_I$ can be nonzero only for $I=(0,..,0)$. It follows that the functions $u_{i,j}$ are constant on the multi-annulus. By proposition \ref{prop:uniq}, $u_{i,j}$ is constant on the whole of $\X$ for all $i$ and $j>0$.   

Thus, we have completely determined the map $\phi^{-1}$: Let us write it using the affine coordinates $t_{i,j}$. There exist constants $a_{i,j}\in K$ such that
$$\phi^{-1}((t_{i,j})_{i,j})=(a_{i,j}t_{\mu(i),j}^{(-1)^{r_i}})_{i,j}.$$
Now, fixing $i$ and $j$, and letting $m=\mu(i)$, for any $t\in K \setminus k_{i}$ there exists a point $x \in \X$ such that $t_{i,j}(x)=t$. It follows that the map $t \mapsto a_{i,j} t^{(-1)^{r_i}}$ is a bijection between $K \setminus k_m$ and $K \setminus k_i$. Since this map is also an automorphism of $K^*$, we get that this map is a bijection between $k_m$ and $k_i$. Hence, $a_{i,j} \in k_i \cap k_m$ and $k_i = k_m$. Finally, no $r_i$ can be equal to $1$,  because by our definition, $r_i$ can only be $1$ when $d_i \geq 2$, and we can always generate $x,y\in K \setminus k_i$ such that $1,x,y$ are linearly independent over $k_i$ but $1,\frac1x, \frac1y$ are dependent. For example we may take $t\in K$ that does not lie in a quadratic extension of $k_i$ (extending $K$ if necessary), and define $x=\frac1t$ and $y=\frac1{t+1}$. This implies that we cannot have an automorphism of $\X_i$ of the form $(t_{i,j})_j \mapsto (\frac{\alpha}{t_{i,j}})_j$, hence $r_i=0$ for all $i$. Therefore, We have reduced $\phi^{-1}$ to a map which permutes coordinates with equal fields, and multiplies each coordinate by an element of the corresponding field. This concludes the proof of theorem \ref{T}.

% **************************  Corollaries ********************

\section{Corollaries of theorem \ref{T}}

\begin{theorem} \label{S}
If $\X=\prod_{i=1}^r\Omega_{k_i,K}^{d_i}$ is isomorphic, as a $K$-analytic space, to $\Y=\prod_{j=1}^s\Omega_{k'_j,K}^{d'_j}$ then $r=s$ and after rearranging the factors we get $k_i = k'_i$ and $d_i=d'_i$. \end{theorem} \begin{proof} Let $\phi: \X \rightarrow \Y$ be an isomorphism. We define an automorphism $\psi$ of $\X \times \Y$ by:
$$ \psi(x,y)=(\phi^{-1}(y), \phi(x)). $$ The theorem now follows immediately from theorem \ref{T}. \end{proof}

\begin{proof}[Proof of theorem \ref{R}] Let $\X_1= \prod_{i=1}^r\Omega_{k_i,K}^{d_i}$ and $\X_2=\prod_{i=1}^s\Omega_{k'_i,K}^{d'_i}$. 
For $i=1,2$, since $\Gamma_i$ is discrete and torsion free, its action on $\X_i$ is discrete and free. This follows from the proof of lemma 6 in \cite{ber} (the proof in \cite{ber} is for the case of a single Drinfeld space, but it holds for products as well, since we still have a proper map from $\X_i$ to the corresponding building, and the stabilizer of a vertex of the building in $\Aut(\X_i)$ is compact). By lemma 4 in \cite{ber}, the quotient $\Gamma_i \setminus \X_i$ exists and the map $p_i:\X_i \rightarrow  \Gamma_i \setminus \X_i$ is an analytic covering map. By \cite{ber}, Theorem 6.1.5, $\X_i$ is contractible. Hence $\X_i$ is the universal covering space of $\Gamma_i \setminus \X_i$. Now, since $\Gamma_1 \setminus \X_1 \cong \Gamma_2 \setminus \X_2$, we have $\X_1 \cong \X_2$. By theorem \ref{S}, we may assume without loss of generality that $\X_1=\X_2$. By lemma 7 in \cite{ber}, the isomorphism from  $\Gamma_1 \setminus \X_1$ to  $\Gamma_2 \setminus \X_2$ is induced by an automorphism $\psi: \X_1 \rightarrow \X_1$. 
For every $\gamma_1 \in \Gamma_1$ and $x\in \X_1$, we have $p_2(\psi \gamma_1 x )=p_2(\psi x )$, hence there exists $\gamma_2 \in \Gamma_2$ such that $\psi \gamma_1 x = \gamma_2 \psi x$. Since the action of $\Gamma_2$ on $\X_1$ is discrete and free, there exists a neighbourhood $U$ of $p_2(x)$, such that $p_2^{-1}(U)$ is a disjoint union of copies of $U$ on which $\Gamma_2$ acts freely. Hence, for $y$ close enough to $x$, we also have $\psi \gamma_1 y = \gamma_2 \psi y$. By proposition \ref{prop:uniq}, $\psi \gamma_1 = \gamma_2 \psi$. This shows that $\Gamma_2= \psi \Gamma_1 \psi ^{-1}$. Using theorem \ref{T}, we may write explicitely the automorphism $\psi$, and get the desired result.\end{proof}

Finally, let us find the automorphisms of $\X$ in the broader category ${\mathcal An}_l$ (see \cite{ber4}, p.30). Recall that the objects of this category are pairs $(L,X)$ where $L$ is a non-Archimedean field over $l$ and $X$ is an $L$-analytic space, and a morphism from $(L_1,X)$ to $(L_2,Y)$ consists of an isometric embedding from $L_2$ to $L_1$ over $l$, and an $L_1$-analytic morphism from $X$ to $Y \hat{\otimes}_{L_2}L_1$. The composition of two morphisms $(\sigma, \phi): (L_1,X_1)\rightarrow (L_2,X_2)$ and $(\tau,\psi): (L_2,X_2)\rightarrow (L_3,X_3)$ is $(\sigma \tau, (\psi \hat{\otimes}_{L_2}L_1) \circ \phi)$. 

\begin{theorem} Let $K,l,r, k_1,..,k_r, d_1,..,d_r$ be as in the introduction. Let $(\sigma, \phi)$ be an automorphism of $\X = \prod_{i=1}^r\Omega_{k_i,K}^{d_i}$ in the category ${\mathcal An}_l$. Then, there exists a permutation $\mu \in S_r$ and $g_i \in PGL_{d_i+1}(k_i)$ for $i=1,..,r$, such that $d_{\mu(i)}=d_i$ and $\sigma(k_i)=k_{\mu(i)}$ for all $i$, and $\phi: \X \rightarrow \X \hat{\otimes}_{K,\sigma}K$ is given by $\phi = \mu \circ (g_1,..,g_r)$, where the action of $\mu$ and $(g_1,..,g_r)$ is as in theorem \ref{T}.\end{theorem}
\begin{proof}We have $\X \hat{\otimes}_{K,\sigma}K = \prod_{i=1}^r\Omega_{\sigma(k_i),K}^{d_i}$. Since $\phi$ is an isomorphism of $K$-analytic spaces, the theorem follows from theorems \ref{S} and \ref{T}. \end{proof}

%\begin{acknowledgements} 
\textbf{Acknowledgements.} The author would like to thank Professor V. Berkovich for suggesting the problem to him, and for valuable discussions and advice. 
%\end{acknowledgements}

\fussy


\begin{thebibliography}{1}
\bibitem[Ber1]{ber}Berkovich, V. G., The automorphism group of the Drinfeld half-plane, C. R. Acad. Sci. Paris S\'er. I Math. {\bf 321} (1995), no. 9, 1127-1132.

\bibitem[Ber2]{ber2} Berkovich, V. G., Spectral theory and analytic geometry over non-Archimedean fields, Mathematical Surveys and Monographs, vol. 33, American Mathematical Society, 1990.

\bibitem[Ber3]{ber3} Berkovich, V. G., Smooth $p$-adic analytic spaces are locally contractible, Invent. Math., {\bf 137} (1999), no. 1, 1-84.

\bibitem[Ber4]{ber4} Berkovich, V. G., \'Etale cohomology for non-Archimedean analytic spaces, Publ. Math. IHES {\bf 78}, 1993, 5-161.

\bibitem[BGR]{bgr} Bosch, S., G\"untzer, U., Remmert, R. : {\bf Non-Archimedean Analysis}, Springer-Verlag, Berlin, 1984.

\bibitem[BT]{brt}Bruhat, F., Tits, J., Groupes r\'eductifs sur un corps local. I. Donn\'ees radicielles valu\'ees, Publ. Math. IHES, {\bf 41}, 1972, 5-251. 

%\bibitem[7]{ds}De Shalit, E., Residues on buildings, and de-Rham cohomology of $p$-adic symmetric domains, \emph{Duke Math. J.}, %{\bf 106} (2000), 123-191.

\bibitem[Gar]{gar}Garrett, P., {\bf Buildings and classical groups}, Chapman \& Hall, London, 1997. 

\bibitem[Hel]{helg} Helgason, S., {\bf Differential Geometry, Lie Groups, and Symmetric Spaces}, Academic Press, New York, 1978.

\bibitem[Kat]{ikato} Ishida, M.-N., Kato, F., The strong rigidity theorem for non-Archimedean uniformization, T\^ohoku Math. J. {\bf 50} (1998), 537-555.

\bibitem[Kap]{kap} Kaplanski, I., {\bf Commutative Rings (Revised ed.)}, University of Chicago Press, Chicago, 1974.

\bibitem[Rap]{rap} Rapoport, M., Zink, Th., Period spaces for $p$-divisible groups, Ann. of Math. Stud. {\bf 141}, Princeton Univ. Press, 1996.

\bibitem[RTW]{rtw} R\'emy, B., Thuillier, A., Werner, A.,  Bruhat-Tits theory from Berkovich's point of view. II: Satake compactifications of buildings, J. Institut Math. Jussieu {\bf 11} (2012), no. 2, 421–465.

\bibitem[SS]{ss} Schneider, P., Stuhler, U., The cohomology of $p$-adic symmetric spaces, Inv. Math. {\bf 105} (1991), 47-122.

\bibitem[Var]{var} Varshavsky, Y., $p$-adic uniformization of unitary Shimura varieties. II, J. Diff. Geometry {\bf 49} (1998), 75-113.
\end{thebibliography}
\end{document}